\documentclass[12pt]{amsart}

\usepackage{fullpage}
\usepackage{amssymb, hyperref}
\usepackage{amsthm}
\numberwithin{equation}{section}
\newtheorem{theorem}{Theorem}[section]

\newtheorem{definition}[theorem]{Definition}

\newtheorem{lemma}[theorem]{Lemma}

\newtheorem{proposition}[theorem]{Proposition}
\newtheorem{remark}[theorem]{Remark}

\newcommand{\sm}{{\rm sm}}
\newcommand\Peaks{{\rm Peaks}}
\newcommand\Spin{{\rm Spin}}
\newcommand\SL{{\rm SL}}
\newcommand\Supp{{\rm Supp}}

\newcommand\bP{\mathbb{P}}
\newcommand\bQ{\mathbb{Q}}

\newcommand\hf{\widehat{f}}
\newcommand\hx{\widehat{x}}

\newcommand\hw{\widehat{w}}

\newcommand\hX{\widehat{X}}

\newcommand\hpi{\widehat{\pi}}

\newcommand\tw{{\widetilde w}}

\newcommand\tX{{\widetilde X}}

\begin{document}

\author{Michel Brion and S.~Senthamarai Kannan}

\title{Some combinatorial aspects of generalised Bott-Samelson varieties}

\begin{abstract}
We obtain two combinatorial results: an equality of Weyl groups and 
an inequality of roots, in the setting of generalised Bott-Samelson
resolutions of minuscule Schubert varieties. These results are 
used in the companion paper \cite{BK} to describe minimal rational
curves on these resolutions, and their relation to lines on the Schubert
varieties.
\end{abstract}

\maketitle

% \noindent
% Keywords:

\section{Introduction}

The generalised Bott-Samelson varieties of the title are certain towers of
locally trivial fibrations with fibers being Schubert varieties. Bott-Samelson
varieties, for which all fibers are projective lines, yield widely used
desingularisations of Schubert varieties. Their generalisations were 
introduced by Sankaran and Vanchinathan (see \cite{SV94, SV95}), 
to construct small resolutions of Schubert varieties in the 
symplectic and orthogonal Grassmannian. They were then systematically 
studied by Perrin (see \cite{NP}); he associated a quiver to any 
minuscule Schubert variety $X$, and he constructed generalised Bott-Samelson
desingularisations of $X$ in terms of this quiver. In particular, Perrin's 
``construction 1'' yields all small resolutions of $X$ (see [loc.~cit., \S 5.4, \S 7.5]).

The motivation for the present paper comes from our investigation of
lines in minuscule Schubert varieties, and minimal rational curves
(an intrinsic version of lines) on their generalised Bott-Samelson 
resolutions, in the companion paper \cite{BK}. The main theorem
of the latter paper (Theorem 4.9) yields a description of the 
families of minimal rational curves on the resolutions obtained by
construction 1. Its proof combines geometric arguments with two 
combinatorial results, which are proved in the present paper. 

We now outline these two results by using notation defined in 
Section \ref{sec:prelim}. Consider a semi-simple, simply-laced 
algebraic group $G$, a minuscule parabolic subgroup $P \subset G$, 
and a Schubert variety $X(w) \subset G/P$. Let 
$\hpi: \hX(\hw) \to X(w)$ be the generalised Bott-Samelson 
desingularisation associated with a generalised reduced decomposition 
$\hw = (w_1,\ldots,w_m)$ obtained by construction 1; then 
$\hX(\hw)$ is equipped with a base point $\hx$ and a locally trivial 
fibration $\hf : \hX(\hw) \to X(w_1)$ with fiber 
$\hX(w_2,\ldots, w_m)$. Moreover, $X(w_1) \simeq G_1/P_1$ 
for a semi-simple, simply-laced subgroup $G_1 \subset G$ 
and a minuscule parabolic subgroup $P_1 \subset G_1$, and 
$G_1$ acts on $\hX(\hw)$ compatibly with its action on $X(w_1)$. 
In loose terms, our first result asserts that $\hx$ is
fixed by a Levi subgroup of $P_1$. This translates into 
an equality of Weyl groups, obtained in Theorems \ref{thm:WMG},
\ref{thm:WMD} and \ref{thm:WME}.

In the above situation, it is easy to show that there is a unique
simple root $\alpha$ such that the root $w_1^{-1}(\alpha)$ is
negative; then $w^{-1}(\alpha)$ is a negative root as well
(see \cite[Lem.~5.1]{BK}). Our second result asserts that
$w^{-1}(\alpha) \leq w_1^{-1}(\alpha)$ with equality if and
only if $w = w_1$, that is, $X(w)$ is smooth. See Propositions 
\ref{prop:heightA}, \ref{prop:heightD} and \ref{prop:heightE}.

We obtain both results via a case-by-case analysis. Types $A_n$ 
and $D_n$ present somewhat different features (see Remark 
\ref{rem:WMD}). The exceptional types $E_6, E_7$ are handled 
via a classification of the corresponding generalised Bott-Samelson 
varieties, which yields many examples of such varieties.

We informed Nicolas Perrin of preliminary versions of our work,
and he came up with shorter, uniform proofs of the equality
of Weyl groups and the root inequality. In turn, we obtained 
a variant of Perrin's proof of the root inequality, which is also 
uniform and perhaps more self-contained, and also a short,
uniform proof of the equality of Weyl groups; these are presented
in the final section of \cite{BK}. We believe that our original, 
case-by-case approach is still of interest,
as it yields additional information (for example, the Weyl groups
under consideration are generated by simple reflections in
type $A_n$, but not in type $D_n$) as well as many examples.

This paper is organized as follows. In Section \ref{sec:prelim}, 
we gather notation and recall some basic facts on flag varieties 
and their Schubert varieties. In Section \ref{sec:perrin}, we first
survey the construction of generalised Bott-Samelson varieties, 
following \cite{NP} and \cite[\S 4.2]{BK}; then we collect some 
auxiliary results. Section \ref{sec:weyl}, containing the proofs of 
the main results, forms the bulk of the paper.

\section{Preliminaries on Schubert varieties}
\label{sec:prelim}

\subsection{Flag varieties}
\label{subsec:flag}

Let $G$ be a simply-connected semi-simple algebraic group 
over an algebraically closed field. Let $P$ be a parabolic 
subgroup of $G$. Let $X=G/P$, and let $x=P/P$ be the base point.

Choose a Borel subgroup $B \subset P$ and
a maximal torus $T \subset B$. Denote by $X^*(T)$
the character group of $T$, and by $R \subset X^*(T)$ 
the root system of $(G,T)$; then the subset $R⁺$ 
of roots of $(B,T)$ is a set of positive roots of $R$. 
Also, denote by $R^-$ the corresponding set of negative roots, 
and by $S = \{ \alpha_1, \ldots, \alpha_n \} \subset R^+$
the set of simple roots. 
Given $\lambda, ~\mu \in X^{*}(T)$, we say that 
$\lambda \geq \mu$ if $\lambda-\mu$ is a nonnegative 
integral linear combination of simple roots. Further, 
$\lambda > \mu$ if in addition $\lambda-\mu\neq 0$.

We also have the coroot system $R^{\vee}$ with simple 
roots $\alpha_1^{\vee}, \ldots,\alpha_n^{\vee}$; these
form a basis of the cocharacter lattice $X_*(T)$.
The dual basis of the character lattice $X^*(T)$
consists of the fundamental weights
$\varpi_1,\ldots, \varpi_n$. More intrinsically,
for any simple root $\alpha$, we will denote by
$\varpi_{\alpha}$ the fundamental weight with
value $1$ at $\alpha^{\vee}$, and $0$ at all
other simple coroots. 

The Weyl group $W = N_G(T)/T$ is generated by the associated
simple reflections $s_1, \ldots,s_n$. We denote by
$\ell$ the corresponding length function on $W$.
A reduced decomposition of $w \in W$ is a sequence 
\[ \tw = (s_{i_1}, s_{i_2}, \ldots, s_{i_k}) \]
of simple reflections such that 
$w = s_{i_1} s_{i_2} \cdots s_{i_k}$ and $\ell(w) = k$.
For any $w \in W$, we denote by $\dot w \in N_G(T)$ a 
representative. Also, for any $\beta \in R$, we denote
by $U_{\beta} \subset G$ the corresponding root 
subgroup. 

Consider the Levi decomposition $P = R_u(P) L$,
where $L$ is a connected reductive subgroup of 
$G$ containing $T$; then $B_L := B \cap L$
is a Borel subgroup of $L$. Denote by 
$R_L \subset R$ the root system of $(L,T)$,
with subset of positive roots $R^+_L = R_L \cap R^+$ 
and subset of simple roots $I := R_L \cap S$. Then $P$ is
generated by $B$ and the $\dot s_{\alpha}$, where
$\alpha \in I$; we write $P = P_I$ and $L = L_I$.
We denote the parabolic subgroup $P_I$ also by 
$P^{S \setminus I}$. With this notation, the parabolic 
subgroup $P$ is maximal if and only if $P = P^{\alpha}$ 
for some simple root $\alpha$. 

Note that $P$ is uniquely determined by the dominant
weight $\varpi := \sum_{i \in I} \varpi_i$. We say that 
$P$ is minuscule, if so is $\varpi$; that is,
$\langle \varpi, \beta^{\vee} \rangle \leq 1$ 
for all $\beta \in R^+$. The minuscule dominant
weights are exactly the sums of minuscule fundamental
weights $\varpi_{\alpha}$ associated with irreducible 
factors of $R$; the corresponding simple roots 
$\alpha$ will be called minuscule as well.

\subsection{Schubert varieties}
\label{subsec:schubert}

We keep the notation and assumptions of the previous
subsection. The Weyl group $W_L = N_L(T)/T$ 
is generated by the simple reflections $s_{\alpha}$, where 
$\alpha \in I$; we also denote this group by $W_{I}$. 
Let $W^I$ denote the subset of $W$ consisting of those
$w$ such that $w(\beta) \in R^+$ for all $\beta \in I$;
equivalently, $R^+_I \subset w^{-1}(R^+)$.
Then $W^I$ is a set of representatives of the coset space
$W/W_I$, consisting of the elements of minimal length in 
their right coset. Note that $w \in W^I$ has length $1$ 
if and only if $w = s_{\alpha}$ for some 
$\alpha \in S \setminus I$. 
On the other hand, the unique element of maximal
length in $W^I$ is $w_0^I = w_0 w_{0,I}$, where $w_0$ 
(respectively, ~$w_{0,I}$) denotes the longest element of $W$ 
(respectively, ~$W_I$).

For any $w \in W$, the point ${\dot w} x \in G/P$ is
independent of the choice of the representative $\dot w$;
we thus denote this point by $wx$. Recall that the $wx$,
where $w \in W^I$, are exactly the $T$-fixed points in $G/P$;
moreover, $G/P$ is the disjoint union of the $B$-orbits 
$B wx$. The stabilizer $B_{wx}$ is generated
by $T$ and the root subgroups $U_{\beta}$, where 
$\beta \in R^+ \cap  w(R^+)$; in particular, 
$B_{wx}$ is smooth and connected. 
The closure of $Bwx$ in $G/P$ is the Schubert variety 
$X(w)$; we have $\dim(X(w)) = \dim(Bwx) = {\ell}(w)$. 
Note that $X(w_0^I) = X$.

We say that the homogeneous space $G/P$ is minuscule if so is $P$. 
Then the Schubert varieties $X(w) \subset G/P$ and the Weyl group 
elements $w \in W^I$ are called minuscule as well. 
Note that every minuscule homogeneous space under $G$ decomposes 
into a product of minuscule homogeneous spaces under simple factors 
of $G$, and every Schubert variety decomposes accordingly. 
Therefore, to study minuscule Schubert varieties, we may assume
that $G$ is simple; then $P = P^{\alpha}$ for some minuscule
simple root $\alpha$. We may further reduce to the case where 
$G$ is simply-laced (see \cite[Rem.~3.4]{NP}). 

From now on, we assume that $G$ is simple and simply-laced.

\section{Generalised Bott-Samelson varieties}
\label{sec:perrin}

\subsection{Generalisations of Bott-Samelson varieties due to Perrin}
\label{subsec:perrin1}

Recall that the set of simple roots $\alpha$ such that $s_{\alpha}$ 
occurs in a reduced decomposition of $w$ is independent of 
the choice of a reduced decomposition, and called the {\it support} 
of $w$. We denote this set by $\Supp(w)$.
The subgroup of $G$ generated by the $U_{\pm \alpha}$,
where $\alpha \in \Supp(w)$, will be denoted by $G_w$; this is
the derived subgroup of the Levi subgroup $L_{\Supp(w)}$, and
hence is a semi-simple subgroup of $G$, normalized by $T$
and containing a representative of $w$. Since $G$ is simply-laced, 
every simple factor of $G_w$ is simply-laced as well.

Let $I^w := \{ \alpha \in S ~\vert~ w(\alpha) \in R^+ \}$.
Then $I^{w}$ is  the largest subset of $S$ such that 
$w \in W^{I^{w}}$. We denote $P_{I^w}$ by $P^w$. 
Consider the associated Schubert variety $X(w) \subset G/P^w$,
and denote by $P_w$ the closed reduced subgroup of $G$ 
consisting of those $g$ such that $ g X(w) = X(w)$. 
Then $P_w$ is a parabolic subgroup of $G$ containing $B$,
and hence we have $P_w = P_{I_w}$, where 
$I_w := \{ \alpha \in S ~\vert~ s_{\alpha} w \leq w \}$; here 
$\leq$ denotes the Bruhat order on $W^{I^w}= W/W_{I^w}$. 

Note that $P_w \cap G_w$ is a parabolic subgroup of $G_w$,
and we have
\begin{equation}\label{eqn:isomorphism} 
X(w) = \overline{P_w w P^w}/P^w \simeq 
\overline{(P_w \cap G_w) w (P^w \cap G_w)}/(P^w \cap G_w)
\subset G_w/(P^w \cap G_w).
\end{equation}
We say that $w \in W$ is minuscule if so is $G/P^w$; then 
one may readily check that $G_w/(P^w \cap G_w)$ is a product 
of minuscule homogeneous varieties associated to simple
factors of $G_w$. 

For any minuscule $w \in W$, we have $X(w)_{\sm} = P_w w x$ 
by the main result of \cite{BP}, where $X(w)_{\sm}$ denotes 
the set of smooth points of $X(w)$. In particular, $X(w)$ is smooth 
if and only if it is homogeneous under $P_w$. Then 
$G_w \subset P_w$  in view of \cite[Lem.~4.8]{BK}.

Next, let $w = w_1 w'$, where $w_1, w' \in W$ satisfy
$P^{w_1} \cap G_{w_1} \subset P_{w'}$; equivalently, we have
$I^{w_1} \cap \Supp(w_1) \subset I_{w'}$. We may then define
\[ \tX(w_1,w') := 
\overline{(P_{w_1} \cap G_{w_1}) w_1 (P^{w_1} \cap G_{w_1})}
\times^{P^{w_1} \cap G_{w_1}} X(w'). \]
This is a projective variety equipped with an action of
$P_{w_1} \cap G_{w_1}$ and an equivariant morphism
\[ f_{w_1,w'} : \tX(w_1,w') \longrightarrow X(w_1), \]
a Zariski locally trivial fibration with fiber $X(w')$.
If in addition $\ell(w) = \ell(w_1) + \ell(w')$, then
$w' \in W^{P^w}$ and hence $P^{w'} \supset P^w$.
Thus, if $P^w$ is maximal and $w' \neq 1$, then $P^{w'} = P^w$. 
As a consequence, we obtain another equivariant morphism
\[ \pi_{w_1,w'} : \tX(w_1,w') \longrightarrow G/P^w. \]
One may check that $\pi_{w_1,w'}$ is birational to its image
$X(w)$; it restricts to an isomorphism above the open orbit 
$Bwx$. Also, note that 
\[ P_{w_1} \cap G_{w_1} \subset P_w \cap G_w. \]

This construction can be iterated, under certain additional
assumptions that are discussed in detail in \cite[\S 5.2]{NP}.
We now present some notions and results from [loc.~cit.]:
a finite sequence $\hw = (w_1,\ldots,w_m)$ of elements of $W$
is called a {\it generalised reduced decomposition of} $w$,
if we have $w = w_1 \cdots w_m$ and 
$\ell(w) = {\ell}(w_1) + \cdots + {\ell}(w_m)$.
Such a decomposition is called {\it good} if in addition 
$w$ is minuscule and we have 
\[ I^{w_i} \cap \Supp(w_i) \subset I_{w_{i+1} \cdots w_m}
\subset w_i^{\perp} \cup \Supp(w_i) \quad (1 \leq i \leq m -1), \]
where $w_{i}^{\perp}$ is the set of simple roots $\alpha$ such that 
$s_{\alpha}$ commutes with $w_{i}$. 

Under these assumptions, $(w_{i+1},\ldots,w_m)$ is a good 
generalised reduced decomposition of $w_{i+1} \cdots w_m$ 
for $i = 1, \ldots, m-1$. Moreover, 
$P^{w_i} \cap G_{w_i} \subset P_{w_{i+1} \cdots w_m}$
for all such $i$. 

Given a good generalised reduced decomposition $\hw$ of $w$, 
we obtain a projective variety $\hX(\hw)$ equipped with an action
of $P_w$, a locally trivial fibration
\[ \hf : \hX(\hw) \longrightarrow X(w_1) \]
with fiber $\hX(w_2,\ldots,w_m)$, and a birational morphism
\[ \hpi : \hX(\hw) \longrightarrow X(w). \]
Also, $\hX(\hw)$ has a base point $\hx$ such that
$\hpi(\hx) = wx$ and $\hf(\hx) = w_1 x_1$ with an 
obvious notation. Moreover, $\hpi$ is 
$P_w$-equivariant in view of \cite[\S 5.1]{NP}.
As a consequence, we have the equality of stabilizers 
$P_{w,\hx} = P_{w,wx}$. Clearly, $P_{w,wx}$ contains 
the maximal torus $T$.

Note that $\hX(\hw)$ is smooth if and only if 
$X(w_1), \ldots, X(w_m)$ are smooth. Then we have
$G_{w_1} \subset P_{w_1}$, and hence 
$G_{w_1} \subset P_w \cap G_w$.
Let $T_{w_1}$ be the neutral component of $T \cap G_{w_1}$;
then $T_{w_1}$ is a maximal torus of $G_{w_1}$. Further, 
we have an isomorphism of Weyl groups 
$W(G_{w_1},T_{w_1}) \simeq W(L_{\Supp(w_1)}, T)$.
This identifies $W(G_{w_1},T_{w_1})$ with a subgroup of $W$.

\begin{lemma}\label{lemMT}
Let $\hw = (w_1,\ldots,w_m)$ be a good generalised decomposition
of $w$ such that $\hX(\hw)$ is smooth.

\begin{enumerate}
\item
For any $1\leq i \leq n$ such that 
$s_{i}\in W(G_{w_{1}}, T_{w_{1}})$, we have 
$s_i X(w) = X(w)$. 
\item
We have the inclusion of stabilisers
$G_{w_1,\hx} \subset G_{w_1, w_1 x_1}$. 
\item 
$T_{w_{1}}$ is a common maximal torus of both   
$G_{w_1,\hx}$ and $G_{w_1, w_1 x_1}$. 
\item We have 
$W(G_{w_1,\hx}, T_{w_1})
\subset W(G_{w_1, w_1x_1}, T_{w_1})
\subset W$. 
\item 
We have 
$W(G_{w_1,\hx}, T_{w_1}) = W(G_{w_1, wx}, T_{w_1}) \subset W$. 
\end{enumerate}
\end{lemma}

\begin{proof}
Proof of (1) follows from the fact that $G_{w_1} \subset P_w \cap G_w$. 

Proof of (2) follows from the fact that the morphism 
$\hf:\hX(\hw) \longrightarrow X(w_{1})$ is 
$G_{w_1}$-equivariant and sends $\hx$ to $w_1 x_1$.

Proof of (3) is easy.

Proof of (4): By (2) and (3), we have 
$N_{G_{w_1,\hx}}(T_{w_1}) \subset N_{G_{w_1,w_1 x_1}}(T_{w_1})$. 
Hence, we obtain
$W(G_{w_1,\hx}, T_{w_1})
\subset W(G_{w_1, w_1 x_1}, T_{w_1})$.
Clearly, $W(G_{w_1, w_1 x_1}, T_{w_1})$ 
is isomorphic to a subgroup of $W(G_{w_1}, T_{w_1})$. 
Therefore, $W(G_{w_1, w_1 x_ 1 }, T_{w_1})$ 
is identified with a subgroup of $W$.

Proof of (5): Since  $G_{w_1} \subset P_w \cap G_w$ and 
$(P_w \cap G_w)_{\hx} = (P_w \cap G_w)_{wx}$, we have 
$G_{w_1, \hx} = G_{w_1, wx}$. 
Thus, the Weyl groups are equal.
\end{proof}

\subsection{The quiver associated to a minuscule element $w$}
\label{subsec:perrin2}

We recall the following definitions and construction 1 of Perrin from \cite{NP}.

Let $P = P_I$ be a minuscule parabolic subgroup of $G$, and $w \in W^I$.
Choose a reduced decomposition
$\tw = (s_{i_{1}}, s_{i_{2}}, \ldots , s_{i_{k}})$ of $w$, and let 
$\beta_j = \alpha_{i_{j}}$ ($1 \leq j \leq k$).

\begin{definition}\label{def:sucpred}
We define the successor $s(i)$ (respectively, the predecessor $p(i)$) 
of an element $i\in [1, k]$ by 
$s(i)=min \{j\in [1,k] : j > i ~ and ~ \beta_{j}=\beta_{i} \}$ 
(respectively, by $p(i)=max \{j\in [1,k]: j < i ~ and ~ \beta_{j}=\beta_{i}\}$).
\end{definition}

\begin{definition}\label{def:quiver}
We denote by $Q_{\tw}$ the quiver whose set of vertices is the set $[1,k]$ 
and whose arrows are given in the following way: there is an arrow from $i$ 
to $j$ if $\langle \beta_{j}^{\vee} , \beta_{i} \rangle \neq 0$ and $i < j < s(i)$ 
(or only $i < j$ if $s(i)$ does not exist).  
\end{definition}

By \cite[Prop.~2.1]{ST}, any two reduced decompositions of $w$ 
differ only by commuting relations. So, the quiver $Q_{\tw}$ does not 
depend on the choice of a reduced decomposition $\tw$ of $w$.
Therefore, we denote this quiver by $Q_w$. 

The quiver comes with a coloration of its vertices by simple roots via the map
$\beta:[1,k]\longrightarrow S$ such that $\beta(j)=\beta_{j}$ for all $1\leq j \leq k$.
 
See \cite [\S 2.1]{NP} for more details.

\begin{definition}\label{def:order}
We denote by $R \subset [1,k] \times [1,k]$ the partial order on the vertices 
of the quiver $Q_w$ generated by the relations $(i,j) \in R$ if there exists 
an arrow from $i$ to $j$.
\end{definition}

\begin{definition}\label{def:peak}
We call {\it peak} any vertex of $Q_w$ minimal for the partial order $R$.
The set of peaks is denoted by $\Peaks(Q_w)$.
\end{definition}

We now explain a way of constructing good generalised reduced decompositions 
of $w$ into a product of minuscule elements $(w_i)_{1\leq i \leq m}$ (see 
\cite[\S 5.4]{NP}).

\begin{definition}\label{def:subquiver}
Let $A \subset \Peaks(Q_{w})$. We denote 
by $Q_w(A)$ the full subquiver of $Q_{w}$ with vertices being those $i$ of 
$Q_w$ such that $(j,i)\notin R$ for all $j \in \Peaks(Q_w)\setminus A$.
\end{definition} 

By \cite [Prop.~5.13]{NP}, the quiver $\widehat{Q_w}(A)$ 
obtained from $Q_{w}$ by removing the vertices of $Q_{w}(A)$ is also 
the quiver of a minuscule Schubert variety.

To construct a partition of the quiver $Q_w$ of a minuscule element $w$ into 
quivers $(Q_{w_i})_{1\leq i \leq m}$ with each $w_i$ is a minuscule element,
it suffices to give a partition of $\Peaks(Q_w)$. 
Indeed, given such partition,
$(A_i)_{1\leq i \leq m}$, we define by induction a sequence 
$(Q_i)_{0\leq i \leq m}$ with $Q_0=Q_w$, $Q_{i+1}=\widehat{Q_i}(A_{i+1})$, 
$Q_{w_{1}}=Q_w(A_{1})$. We then denote by $Q_{w_{i}}$ the quiver 
$Q_{i-1}(A_{i})$. The quivers $(Q_{w_{i}})_{1\leq i \leq m}$ form a partition 
of $Q_{w}$. Each quiver $Q_{w_{i}}$ is associated to a minuscule element 
$w_i$. The generalised reduced decomposition 
$\hw=(w_1, w_2, \ldots , w_m)$ is good (see \cite[Prop.~5.15]{NP}).

{\it Construction 1 of Perrin}. Choose any order $\{ i_1, i_2, \ldots ,i_m \}$ 
on $\Peaks(Q_w)$ and set $A_j=\{i_j\}$ ($1\leq j \leq m$).

Throughout this paper, we only consider generalised reduced
decompositions obtained by construction 1. Also, we use repeatedly
Perrin's smoothness criterion: with the above notation, $\hX(\hw)$
is smooth if and only if, for $1 \leq j \leq m$, the simple root
$\beta(i_j)$ is minuscule in $R_{\Supp(w_j)}$
(see \cite[Thm.~7.11]{NP}).

\begin{lemma}\label{lemsr}
Let $\hw=(w_1, w_2, \ldots , w_m)$ 
be a generalised reduced decomposition of $w$ 
obtained by construction 1. Then any simple reflection 
$s_i \in W(G_{w_1, w_1 x_1}, T_{w_1})$ 
lies in $W(G_{w_1, wx}, T_{w_1})$. 
\end{lemma}

\begin{proof}
Since $s_i \leq w_{1}$ and $s_{i}w_{1}x_{1}=w_{1}x_{1}$, we have 
$\alpha_{i}\neq \beta(p)$ where $\Peaks(Q_{w_{1}})=\{p\}$. 
On the other hand, by construction 1, we have 
$\Peaks(Q_{w})\cap Q_{w_{1}}=\{p\}$. Therefore, we have 
${\ell}(s_{i}w)={\ell}(w)+1$.  
Also, by Lemma \ref{lemMT}(1), we have $s_i X(w)=X(w)$. 
Therefore, combining these two together, we obtain $s_i w x = w x$. 
\end{proof}

Let $i_1 < i_2 < \ldots < i_m$ be the ordering of $\Peaks(Q_{w})$ 
induced by the standard increasing ordering of integers. Let 
$\hw=(w_1, w_2, \ldots , w_m)$ be the generalised reduced 
decomposition of $w$ obtained by construction 1
corresponding to this ordering of $\Peaks(Q_{w})$.  

Let $\preceq$ be another ordering of $\Peaks(Q_{w})$.  
Let $\hw'=(w'_1, w'_2, \ldots , w'_m)$ be the generalised reduced 
decomposition of $w$ obtained by construction 1 corresponding 
to this ordering of $\Peaks(Q_{w})$. Let $1\leq q \leq m$ be 
the integer such that $i_{q}$ is the first peak in this ordering. 
Further, assume that $q\neq 1$.

Then, we have 

\begin{lemma}\label{lem:commute} 
Every simple reflection 
$s_{e} \leq w'_{1}$ commutes with every simple 
reflection $s_{f} \leq w_r$ for all $1 \leq r \leq q -1$. 
\end{lemma}

\begin{proof}
By construction 1, for any vertex $l$ of 
$Q_{w}(\Peaks(Q_{w})\setminus \{i_q\})$ 
and for any vertex of $i$ of $Q_{w}(\{i_q\})$, we have $(l,i)\notin R$. 
In particular, if $l < i$ such that $\beta(l)\neq \beta(i)$,  then
we have $\langle \beta(l)^{\vee} , \beta(i) \rangle =0$. 
This implies the assertion. 
\end{proof}

%\begin{lemma}\label{lem:us} 
%Let $\hw=(w_1, w_2, \ldots, w_m)$ be a generalised reduced 
%decomposition of $w$ obtained by construction 1. Then there is 
%a unique simple root $\alpha$ such that $w_{1}^{-1}(\alpha)$ 
%is a negative root.
%\end{lemma}
%
%\begin{proof}
%Let $i_1 \preceq  i_2 \preceq \cdots \preceq i_{\ell}$ 
%be the ordering of $\Peaks(Q_w)$ that induces the given 
%generalised reduced decomposition.
%By construction 1, we have $\Peaks(Q_{w}) \cap Q_w(\{i_1\}) = \{i_1\}$. 
%On the other hand, we have $\beta(\Peaks(Q_w))=S \cap R^+(w^{-1})$, 
%where for any $v \in W$, we denote by $R^+(v)$ the set of positive roots
%$\gamma$ such that $v(\gamma)$ is a negative root. 
%Thus, we have $S\cap R^+(w_1^{-1}) = \{ \beta(i_1) \}$.
%\end{proof}

\section{Equality of Weyl groups and a root inequality}
\label{sec:weyl}

Let $G$ be as in Section \ref{sec:prelim}. Throughout this section, 
we consider a minuscule parabolic subgroup $P = P_I$,
a Weyl group element $w\in W^I$, and a generalised reduced 
decomposition $\hw =(w_1, w_2, \ldots, w_m)$ of $w$, obtained 
by construction 1 of Perrin.

In view of (\ref{eqn:isomorphism}),
we see that $X(w)$ is a Schubert subvariety of a minuscule flag variety 
for $G_w$. Thus, we may assume that $G_w=G$; then $\Supp(w)=S$.

\subsection{Equality of Weyl groups in type $A_n$}
\label{subsec:weylA}

Let $G$ be of type $A_n$, that is, $G = \SL_{n+1}$; then every 
fundamental weight is minuscule and hence the minuscule parabolic 
subgroups are exactly the maximal ones. The aim of this subsection 
is to prove the following:

\begin{theorem}\label{thm:WMG}
For any $w$ and $\hw$ as above, we have 
\[ W(G_{w_1,\hx}, T_{w_1}) = W(G_{w_1, w_1 x_1}, T_{w_1}) \]
and this group (viewed as a subgroup of $W$) is generated by simple reflections. 
\end{theorem}

To prove this result, we first set up notation. We order the simple roots as in 
\cite[p.58]{HU}. Let $P:= P^w$; then $P = P^{\alpha_r}$ for a unique integer 
$1\leq r \leq n$.
Further, $w(i) < w(i+1)$ for every $1\leq i \leq r-1$. Let $a_{i}=w(i)-1$. 
Then $i\leq a_{i} < a_{i+1}\leq n$ for every $1\leq i \leq r-1$. Let 
\[ v=(s_{a_{1}}s_{a_{1}-1}\cdots s_{1}) (s_{a_{2}}s_{a_{2}-1}\cdots s_{2}) 
\cdots (s_{a_{r}}s_{a_{r}-1}\cdots s_{r}). \] 
Note that $v\in W^{S\setminus \{\alpha_{r}\}}$.

Since $w, v \in W^{S \setminus \{\alpha_{r}\}}$ and $w(i)=v(i)$ for all $1\leq i \leq r$, 
we have  
\[ w= (s_{a_{1}}s_{a_{1}-1}\cdots s_{1})(s_{a_{2}}s_{a_{2}-1}\cdots s_{2}) 
\cdots (s_{a_{r}}s_{a_{r}-1}\cdots s_{r}). \]
For integers $b_{i}\leq a_{i}$, we let $w_{b_{i}, a_{i}}=s_{a_{i}}s_{a_{i}-1}\cdots s_{b_{i}}$. 
Since 
\[ \tw = (s_{a_{1}}, s_{a_{1}-1}, \ldots , s_{1}, s_{a_{2}}, s_{a_{2}-1}, 
\ldots , s_{2},  \ldots , s_{a_{r}}, s_{a_{r}-1}, \ldots , s_{r}) \] 
is a reduced decomposition of $w$, 
we see that $(w_{1, a_{1}}, w_{2, a_{2}}, \ldots , w_{r, a_{r}})$ is a generalised 
reduced decomposition of $w$.

Next, we prove a succession of preliminary results:

\begin{lemma}\label{lemcombA}
Let $1\leq i < k \leq r$. 
Let $b_{s}$ ($i\leq s \leq k$) be a sequence of integers such that 
$s\leq b_{s} \leq a_{s}$ and $b_{s+1}= 1+b_{s}$ for all $i\leq s \leq k-1$.
Let $w'=w_{b_{i}, a_{i}}w_{b_{i+1}, a_{i+1}}\cdots w_{b_{k}, a_{k}}.$ 
Then for any integer $b_{i}\leq l \leq a_{i}-1$, we have $s_{l}w'=w's_{j}$ 
for some integer $j\neq b_{k}$ such that $s_{j}\leq w'$.
\end{lemma}

\begin{proof}
Since $b_{s+1}= 1+b_{s}$ for all $i\leq s \leq k-1$ the pattern of 
$w_{b_{i}, a_{i}}w_{b_{i+1}, a_{i+1}}\cdots w_{b_{k}, a_{k}}$ is similar to that of 
$w_{i, a_{i}}w_{i+1, a_{i+1}}\cdots w_{k, a_{k}}$. Thus, we may assume that 
$b_{s}=s$ for all $i\leq s \leq k$.   

Since $s_{l}s_{m}=s_{m}s_{l}$ for all $m\geq l+2$,  
$s_{l}s_{l+1}s_{l}=s_{l+1}s_{l}s_{l+1}$ and $s_{l+1}s_{t}=s_{t}s_{l+1}$ 
for all $i\leq t\leq l-1$, we have 
$s_{l}w_{i, a_{i}}=s_{l}s_{a_{i}}s_{a_{i}-1}\cdots s_{i}
=s_{a_{i}}\cdots s_{i}s_{l+1}.$

Therefore, we have 
$s_{l}w'=w_{i, a_{i}}s_{l+1}w_{i+1, a_{i+1}}\cdots w_{k, a_{k}}$. Further, since 
$i+1\leq l+1 \leq a_{i+1}-1$, by induction on $k-i$, we have 
$s_{l+1}w_{i+1, a_{i+1}}\cdots w_{k, a_{k}}=w_{i+1 , a_{i+1}}\cdots w_{k, a_{k}}s_{j}$ 
for some integer $j\neq k$ such that $s_{j}\leq w'$.

Hence, we have $s_{l}w'=w_{i,a_{i}}w_{i+1, a_{i+1}}\cdots w_{k, a_{k}}s_{j}=w's_{j}$.
\end{proof}

\begin{lemma}\label{lemcA}
Let $i$, $k$, $a_{s}$, $b_{s}$ ($i\leq s \leq k$) and $w'$ be as in Lemma 
\ref{lemcombA}. Further, assume that $a_{s+1}=1+a_{s}$ for all $i\leq s \leq k-1$. 
Then for any integer $b_{i}\leq l \leq a_{k}$ different from $a_{i}$, we have 
$s_{l}w'=w's_{j}$ for some integer $j\neq b_{k}$ such that $s_{j}\leq w'$.
\end{lemma}

\begin{proof}
As in the proof of Lemma \ref{lemcombA}, we may assume that $b_{s}=s$ 
for all $i\leq s \leq k$.   

If $l \leq a_{i}-1$, then, we are done by Lemma \ref{lemcombA}.

If $l=a_{i+1}$, we first show that 
$s_{l}w_{i, a_{i}}w_{i+1, a_{i+1}}=w_{i,a_{i}}w_{i+1,a_{i+1}}s_{i}$.

For the base case $a_{i}=i$ and $a_{i+1}=i+1$, we have 
\[ s_{l}w_{i, a_{i}}w_{i+1, a_{i+1}}=s_{i+1}s_{i}s_{i+1}=s_{i}s_{i+1}s_{i}=
w_{i, a_{i}}w_{i+1, a_{i+1}}s_{i}. \] 

If $a_{i}\geq i+1$, then we obtain 
\[ s_{l}w_{i, a_{i}}w_{i+1, a_{i+1}}
= s_{a_{i}+1}s_{a_{i}}s_{a_{i}+1} (s_{a_{i}-1}\cdots s_{i})(s_{a_{i}}\cdots s_{i+1}) \]
\[ = s_{a_{i}}s_{a_{i}+1}(s_{a_{i}}((s_{a_{i}-1}\cdots s_{i})(s_{a_{i}}\cdots s_{i+1})))
=s_{a_{i}}s_{a_{i}+1}(w_{i,a_{i}-1}w_{i+1,a_{i}}s_{i})=w_{i,a_{i}}w_{i+1,a_{i+1}}s_{i} \]
by induction on length. 

Since $s_{i}w_{s, a_{s}}=w_{s,a_{s}}s_{i}$ for all $s \geq i+2$, this yields 
\[ s_{l} w'=w_{i, a_{i}}w_{i+1, a_{i+1}}s_{i}w_{i+2, a_{i+2}}\cdots w_{k, a_{k}}
= w' s_i. \]

The proofs for the cases $l=a_{s}$, $i+2 \leq s \leq k$ are similar.
\end{proof} 

\begin{lemma}\label{lemcomb2A}
Let $i$, $k$, $a_{s}$, $b_{s}$ ($i\leq s \leq k$) and $w'$ be as in Lemma \ref{lemcA}. 
Then for any integer $b_{k}+1\leq t \leq a_{k}$, we have $w's_{t}=s_{l}w'$ for some 
integer $l\neq a_{i}$ such that $s_{l}\leq w'$.
\end{lemma}

\begin{proof}
Again, we may assume that $b_{s}=s$ for all $i\leq s \leq k$.   

Since, $s_{t}s_{m}=s_{m}s_{t}$ for all $m\leq t-2$,  $s_{t}s_{t-1}s_{t}=s_{t-1}s_{t}s_{t-1}$ and $s_{j}s_{t-1}=s_{t-1}s_{j}$ for all 
$j\geq t+1$, we have 
$w_{k, a_{k}}s_{t}=s_{a_{k}}s_{a_{k}-1}\cdots s_{k}s_{t}=s_{t-1}s_{a_{k}}\cdots s_{k}$.
Therefore, we have 
$w's_{t}=w_{i, a_{i}}w_{i+1, a_{i+1}}\cdots w_{k-1, a_{k-1}}s_{t-1} w_{k, a_{k}}$. 
Further, since $k \leq t-1 \leq a_{k}-1=a_{k-1}$, by induction on $k-i$, we have 
$w_{i, a_{i}}\cdots w_{k-1, a_{k-1}}s_{t-1}
= s_{l}w_{i , a_{i}}\cdots w_{k-1, a_{k-1}}$ for some integer 
$l\neq a_{i}$ such that $s_{l}\leq w'$.

Hence, we obtain
$w's_{t}=s_{l}w_{i,a_{i}}w_{i+1, a_{i+1}}\cdots w_{k, a_{k}}=s_{l}w'$.
\end{proof}

\begin{lemma}\label{lemc2A}
Let $i$, $k$, $a_{s}$, $b_{s}$ ($i\leq s \leq k$) and $w'$ be as in Lemma \ref{lemcA}. 
Then for any integer $b_{i}\leq t \leq a_{k}$ different from $b_{k}$, we have $w's_{t}=s_{l}w'$ 
for some integer $l\neq a_{i}$ such that $s_{l}\leq w'$.
\end{lemma}

\begin{proof}
As in the proof of Lemma \ref{lemcombA}, we may assume that 
$b_{s}=s$ for all $i\leq s \leq k$.   

If $t\geq k+1$, then we are done by Lemma \ref{lemcomb2A}.

If $t=k-1$, we first show that 
$w_{k-1, a_{k-1}}w_{k, a_{k}}s_{t}=s_{a_{k}}w_{k-1,a_{k-1}}w_{k,a_{k}}$.

For the base case $k=a_{k}$ and $k-1=a_{k-1}$, we have 
\[ w_{k-1,a_{k-1}}w_{k,a_{k}}s_{t}=s_{k-1}s_{k}s_{k-1}
=s_{k}s_{k-1}s_{k}=s_{a_{k}}w_{k-1,a_{k-1}}w_{k,a_{k}}. \]

If $a_{k-1}\geq k$, then we obtain
\[ w_{k-1, a_{k-1}}w_{k, a_{k}}s_{t}
=w_{k,a_{k-1}}w_{k+1,a_{k}}(s_{k-1}s_{k}s_{k-1})
=w_{k,a_{k-1}}w_{k+1,a_{k}}(s_{k}s_{k-1}s_{k}) \]
\[ =(w_{k,a_{k-1}}w_{k+1,a_{k}}s_{k})(s_{k-1}s_{k})
=(s_{a_{k}}w_{k,a_{k-1}}w_{k+1,a_{k}})(s_{k-1}s_{k})=s_{a_{k}}w_{k-1,a_{k-1}}w_{k,a_{k}} \]
by induction on length.

Since $w_{j,a_{j}}s_{a_{k}}=s_{a_{k}}w_{j,a_{j}}$ for all 
$i\leq  j\leq k-2$, we have $w's_{t}=s_{a_{k}}w'$.

The proofs for the cases $i \leq t \leq k-2$ are similar.

\end{proof} 

Recall that 
$\tw = 
(s_{a_1}, s_{a_1 -1}, \ldots , s_1, s_{a_2}, \ldots, s_2, \ldots, s_{a_r}, \ldots, s_r)$ 
is a reduced decomposition of $w=w_{1,a_{1}} w_{2,a_{2}} \cdots w_{r,a_{r}}$. 
Let $\beta_{1}=\alpha_{a_{1}}$, $\beta_{2}=\alpha_{a_{1}-1}$, etc. 
Therefore, the {\it peaks} of the quiver 
$Q_w$ are the indices $i$ such that ${\ell}(s_{\beta_{i}}w)={\ell}(w)-1$.

Let $J(w):=\{2\leq j \leq r: a_{j}-a_{j-1}\geq 2\}\bigcup \{1\}$ and 
let $|J(w)|=m$. By the above paragraph, peaks of the quiver $Q_{w}$ 
are indexed by the elements of $J(w)$. So, $\Peaks(Q_{w})$ is identified with 
$\{i_{j}\in [1,\ell(w)] : j\in J(w)\}$. 
Then we have $\beta(i_{j})=\alpha_{a_{j}}$ for all $j\in J(w)$. 

Let $1=j_{1} < j_{2} < \cdots < j_{m}\leq r$ be the standard increasing ordering 
of elements of $J(w)$, that is; the ordering induced by the usual ordering of 
positive integers. Let $\hw = (w_1, w_2, \ldots , w_m)$ be the 
generalised reduced decomposition of $w$ obtained by construction 1 
corresponding to this ordering.  

By this construction , we have 
$w_{1}=w_{1, a_{1}}w_{2, a_{2}}\cdots w_{j_{2}-1, a_{j_{2}-1}}$. 
Note that $\ell(s_{i}w_{1})={\ell}(w_{1})-1$ if and only if $i=a_{1}$. 
Therefore, again by construction 1, we have 
$\Peaks(Q_{w_{1}})=\{p\}$, where $\beta(p)=\alpha_{a_{1}}$. Thus, we have 
$\langle \alpha_{a_{t+1}}^{\vee}, \alpha_{a_{t}} \rangle \neq 0$ for all 
$1\leq t \leq j_{2}-2$. Hence, we have 
$a_{t+1}=1+a_{t}$ for all $1\leq t \leq j_{2}-2$.

We may now prove Theorem \ref{thm:WMG} in the case where the generalised 
reduced decomposition is associated to the above standard ordering of peaks. 
For the reader's convenience, we recall its statement:

\begin{lemma}\label{lemWM}
With the above assumptions, we have
$W(G_{w_1,\hx}, T_{w_1}) = W(G_{w_1, w_1 x_1}, T_{w_1})$ 
and this group is generated by simple reflections. 
\end{lemma}

\begin{proof}
By Lemma \ref{lemMT}(5), we have
$W(G_{w_1,\hx}, T_{w_1}) = W(G_{w_1,wx}, T_{w_1})$.
Further, by Lemma \ref{lemMT}(4), we have 
$W(G_{w_1, \hx}, T_{w_1})
\subset W(G_{w_1, w_1 x_1}, T_{w_1})$.
Thus, it suffices to prove that
$W(G_{w_1, w_1 x_1}, T_{w_1}) \subset W(G_{w_1,wx}, T_{w_1})$
and the latter is generated by reflections.

Recall that $w_{1}=w_{1, a_{1}}w_{2, a_{2}}\cdots w_{j_{2}-1, a_{j_{2}-1}}.$
We first prove the following

{\it Claim}. $W(G_{w_1, w_ 1 x_ 1 }, T_{w_1})$ 
is generated by simple reflections.

Let $v\in W(G_{w_1}, T_{w_1})$ be such that $vw_{1}=w_{1}u$ 
for some $u\in W(G_{w_1}, T_{w_1})$. We prove by induction on 
${\ell}(u)$ that $v$ is a product of simple reflections $s_{j}\leq w_{1}$ such that 
$s_{j}w_{1}=w_{1}s_{t}$ for some integer $t\neq j_{2}-1$ with  $s_{t}\leq w_{1}$. 

If ${\ell}(u)=1$, then we have $u=s_{t}$ for some $t\neq j_{2}-1$ such that 
$s_{t}\leq w_{1}$. Hence by Lemma \ref{lemc2A}, we have $v=s_{l}$ 
for some integer $l\neq a_{1}$ such that $s_{l}\leq w_{1}$. 

So, assume that ${\ell}(u)\geq 2$. Choose an integer $t\neq j_{2}-1$ such that 
$s_{t}\leq w_{1}$ and ${\ell}(us_{t})={\ell}(u)-1$. 

We have $vs_{w_{1}(\alpha_{t})}w_{1}=vw_{1}s_{t}=w_{1}us_{t}$. Note that 
$w_{1}(\alpha_{t})\in R^{+}\setminus  v^{-1}(R^{+})$. 
Therefore we have $s_{w_{1}(\alpha_{t})}\in W(G_{w_{1}}, T_{w_{1}})$. Since 
${\ell}(us_{t})={\ell}(u)-1$, by induction $vs_{w_{1}(\alpha_{t})}$ is a product 
of simple reflections $s_{j}\leq w_{1}$ such that $s_{j}w_{1}=w_{1}s_{p}$ 
for some integer $p\neq j_{2}-1$ such that $s_{p}\leq w_{1}$.

On the other hand, by Lemma \ref{lemc2A}, we have $w_{1}s_{t}=s_{l}w_{1}$ 
for some integer $l\neq a_{1}$ such that $s_{l}\leq w_{1}$. Therefore, we have  
$w_{1}(\alpha_{t})=\alpha_{l}$ and so $v$ is a product of simple reflections 
$s_{j}\leq w_{1}$ such that $s_{j}w_{1}=w_{1}s_{p}$ for some integer $p\neq j_{2}-1$ 
such that $s_{p}\leq w_{1}$.

This proves the claim.

Now, let $v\in W(G_{w_1, w_1 x_1}, T_{w_1})$. 
Then by the claim, $v$ is a product of simple reflections $s_{j}\leq w_{1}$ 
such that $s_{j}w_{1}=w_{1}s_{p}$ for some integer $p\neq j_{2}-1$ such that 
$s_{p}\leq w_{1}$.
On the other hand,  by Lemma \ref{lemsr}, for any such $s_{j}$, we have 
$s_{j}wx=wx$ and hence we have $vwx=wx$. This proves the desired inclusion.

%By Lemma \ref{lemcombA} and Lemma \ref{lemcA} the integers $j$  for which $s_{j}\leq w_{1}$ and $s_{j}w_{1}x_{1}=w_{1}x_{1}$ are precisely the integers $1\leq j \leq a_{j_{2}-1}$ such that $j\neq a_{1}$. Further, for any such $j$, we have $s_{j}w_{1}=w_{1}s_{l}$ for some integer $1\leq l \leq a_{j_{2}-1}$ such that $l\neq j_{2}-1$. Further, for any such $l$, we have $s_{l}w_{j_{2}, a_{j_{2}}}\cdots w_{r, a_{r}}=w_{j_{2}, a_{j_{2}}}\cdots w_{r, a_{r}}s_{t}$ for some $t\neq r$ (see Lemma \ref{lemcombA}). Therefore, we have 
%$s_{i}w=w ~ mod ~ P^{\alpha_{r}}$. Hence, we have $s_{i}wx=wx$.
\end{proof}

To complete the proof of Theorem \ref{thm:WMG}, we now consider 
an arbitrary ordering $\preceq$ of the elements of $J(w)$ and let $j\in J(w)$ 
be the first element in this ordering. 
Let $\hw'=(w'_{1}, w'_{2}, \ldots , w'_{m})$ be the generalised 
reduced decomposition obtained by construction 1 corresponding 
to this ordering of $\Peaks(Q_{w})$. 
Let $1\leq q \leq m$ be the integer such that $j=j_{q}$. Again by 
construction 1, we have $w_{q}=w'_{1}v$ for some $v\in W$ such that 
${\ell}(w_{q})={\ell}(w'_{1})+{\ell}(v)$.

Let $\hX(\hw')$ be the variety corresponding to $\hw'$, with base
point $\hx'$.

Recall from Lemma \ref{lemMT} that 
$W(G_{w'_1,\hx'}, T_{w'_1}) \subset W(G_{w'_1, w'_1 x'_1}, T_{w'_1})
\subset W$. 

By arguing as in the beginning of the proof of Lemma \ref{lemWM}, 
it suffices to prove the inclusion 
$W(G_{w'_1,w'_1 x'_1}, T_{w'_1}) \subset W(G_{w'_1, wx}, T_{w'_1})$. 

The rest of the argument is also similar to that of Lemma \ref{lemWM}.
For completeness of proof, we give the details.

If the ordering of the peaks of $Q_{w}$ is such that $w_{1}'=w_{1}$, 
we are done. Otherwise, let $k=j_{q+1}-1$. Then by construction 1, 
there exists a sequence $l\leq b_{l} \leq a_{l}$ ($j \leq l \leq k$) 
of positive integers such that $b_{l+1}=b_{l}+1$ for all $j\leq l \leq k-1$, 
and $w'_{1}=w_{b_{j},a_{j}}w_{b_{j+1},a_{j+1}}\cdots w_{b_{k},a_{k}}$. 
Further, we have $a_{l+1}=a_{l}+1$ for all $j\leq l \leq k-1$. 

So, $w'_{1}$  satisfies the hypothesis of Lemma \ref{lemc2A}. Therefore, 
we can imitate the proof of Lemma \ref{lemWM}. We first prove the following

{\it Claim :} $W(G_{w'_1,w'_1 x'_1}, T_{w'_1})$ 
is generated by simple reflections.

Let $v\in W(G_{w'_1}, T_{w'_1})$ be such that 
$vw'_{1}=w'_{1}u$ for some $u\in W(G_{w'_{1}}, T_{w'_1})$. 
We prove by induction on $\ell(u)$ that $v$ is a product of simple reflections 
$s_{i}\leq w'_{1}$ such that $s_{i}w'_{1}=w'_{1}s_{t}$ for some integer 
$t\neq b_{k}$ with  $s_{t}\leq w_{1}$. 

If ${\ell}(u)=1$, then we have $u=s_{t}$ for some $t\neq b_{k}$ such that 
$s_{t}\leq w'_{1}$. Hence by Lemma \ref{lemc2A}, we have $v=s_{l}$ 
for some integer $l\neq a_{j}$ such that $s_{l}\leq w'_{1}$. 

So, assume that ${\ell}(u)\geq 2$. Choose an integer $t\neq b_{k}$ such that 
$s_{t}\leq w'_{1}$ and ${\ell}(us_{t})={\ell}(u)-1$. 

We have $vs_{w'_{1}(\alpha_{t})}w'_{1}=vw'_{1}s_{t}=w'_{1}us_{t}$. Note that 
$w'_{1}(\alpha_{t})\in R^{+}\setminus  v^{-1}(R^{+})$. 
Therefore we have $s_{w'_{1}(\alpha_{t})}\in W(G_{w'_{1}}, T_{w'_{1}})$. 
Since ${\ell}(us_{t})={\ell}(u)-1$, by induction $vs_{w'_{1}(\alpha_{t})}$ 
is a product of simple reflections 
$s_{i}\leq w'_{1}$ such that $s_{i}w'_{1}=w'_{1}s_{p}$ for some integer 
$p\neq b_{k}$ such that $s_{p}\leq w'_{1}$.

On the other hand, by Lemma \ref{lemc2A}, we have $w'_{1}s_{t}=s_{l}w'_{1}$ 
for some integer $l\neq a_{j}$ such that $s_{l}\leq w'_{1}$. Therefore, we have  
$w'_{1}(\alpha_{t})=\alpha_{l}$ and so $v$ is a product of simple reflections 
$s_{i}\leq w'_{1}$ such that $s_{i}w'_{1}=w'_{1}s_{p}$ for some integer
$p\neq b_{k}$ such that $s_{p}\leq w'_{1}$.

This proves the claim.

Now, let $v\in W(G_{w'_1, w'_1 x'_1}, T_{w'_1})$. 
Then by the claim, $v$ is a product of simple reflections $s_{i}\leq w'_{1}$ 
such that $s_{i}w'_{1}=w'_{1}s_{p}$ for some integer $p\neq b_{k}$ such that 
$s_{p}\leq w'_{1}$.

On the other hand,  by Lemma \ref{lemsr}, for any such $s_{i}$, we have 
$s_{i}wx=wx$ and hence $vwx=wx$. This completes the proof
of Theorem \ref{thm:WMG}.

\subsection{Root inequality in type $A_n$}
\label{subsec:rootA}

We keep the notation of Subsection \ref{subsec:weylA}. In particular, 
$w\in W$ denotes a minuscule element, and $P:= P^w = P^{\alpha_r}$ 
for a unique integer $1\leq r \leq n$. Recall that there exists 
a unique sequence $1 \leq a_{1} < a_{2} < \ldots < a_{r} \leq n$ of integers 
such that $(w_{1, a_{1}}, w_{2, a_{2}}, \ldots , w_{r, a_{r}})$ is 
a generalised reduced decomposition of $w$. 
Also, recall that $J(w):=\{2\leq j \leq r: a_{j}-a_{j-1}\geq 2\}\bigcup \{1\}$ 
is identified with $\Peaks(Q_{w})$. 

Let $\preceq$ be an arbitrary ordering of the elements of $J(w)$.
Let $\hw' = (w'_1, w'_2, \ldots , w'_m)$ be the generalised reduced 
decomposition of $w$ obtained by construction 1 corresponding to this ordering.
 
Then we have 

\begin{proposition}\label{prop:heightA}
Let $\alpha$ be the unique simple root such that $(w'_{1})^{-1}(\alpha)$ 
is a negative root (see \cite[Lem.~5.1]{BK}). Then we have 
$(w'_{1})^{-1}(\alpha) \geq  w^{-1}(\alpha)$. 
Further, $(w'_{1})^{-1}(\alpha)=w^{-1}(\alpha)$ if and only if $w = w'_{1}$; 
that is, $m=1$.
\end{proposition}

\begin{proof}
We first claim that the above proposition holds for the standard increasing 
ordering 
$1=j_{1} < j_{2} < \cdots < j_{m}\leq r$ of the elements of $J(w)$.

Let $\hw = (w_1, w_2, \ldots , w_m)$ be the 
generalised reduced decomposition of $w$ obtained by construction 1 
corresponding to this ordering. By this construction, we have 
$w_{1}=w_{1, a_{1}}w_{2, a_{2}}\cdots w_{j_{2}-1, a_{j_{2}-1}}$. 
Therefore, we have $\alpha=\alpha_{a_{1}}$. Further, we have 
\[ w_{1}^{-1}(\alpha_{a_{1}})=-\sum_{i=1}^{a_{j_{2}-1}}\alpha_{i}. \]
On the other hand, we have 
\[ w^{-1}(\alpha_{a_{1}})=-\sum_{i=1}^{a_{j_{2}-1}+r+1-j_{2}}\alpha_{i}. \]
Thus, we have $w_{1}^{-1}(\alpha) \geq  w^{-1}(\alpha)$.
Also, if $m\neq 1$, then we have $j_{2}\leq r$ and hence we have 
$r+1-j_{2}\geq 1$.
Hence we have $w_{1}^{-1}(\alpha) >  w^{-1}(\alpha)$, proving the claim. 

We now prove the proposition for an arbitrary ordering of the elements 
of $J(w)$.

let $j\in J(w)$ be the first element in this ordering. Let $1\leq q \leq m$ 
be the integer such that $j=j_{q}$. By the claim, we may assume that 
$q\neq 1$. By construction 1, we have $w_{q}=w'_{1}v$ for some $v\in W$ 
such that ${\ell}(w_{q})={\ell}(w'_{1})+{\ell}(v)$.

Let $k=j_{q+1}-1$. Then by construction 1, there exists a sequence 
$l\leq b_{l} \leq a_{l}$ ($j \leq l \leq k$) of positive integers such that 
$b_{l+1}=b_{l}+1$ for all $j\leq l \leq k-1$, and 
$w'_{1}=w_{b_{j},a_{j}}w_{b_{j+1},a_{j+1}}\cdots w_{b_{k},a_{k}}$. 
Further, we have $a_{l+1}=a_{l}+1$ for all $j\leq l \leq k-1$. 
By Lemma \ref{lem:commute}, 
$w'_{1}$ commutes with $w_{l}$ for all $1\leq l \leq q-1$.
Therefore, we have $b_j \geq j+1$.

Now, we have $\alpha=\alpha_{a_{j}}$. Therefore, 
\[ (w'_{1})^{-1}(\alpha_{a_{j}})= 
-\sum_{i=b_{j}}^{a_{k}}\alpha_{i}. \]  
On the other hand, we have 
\[ w_{q}^{-1}(\alpha_{a_{j}}) = -\sum_{i=j}^{a_{k}}\alpha_{i}. \]

Since $b_{j}\geq j+1$, we obtain
\begin{equation}\label{eqn:w'w} 
(w'_{1})^{-1}(\alpha_{a_{j}}) > w_{q}^{-1}(\alpha_{a_{j}}). 
\end{equation} 
Applying the claim to the generalised reduced decomposition 
$\hat{v}=(w_{q}, w_{q+1}, \cdots , w_{m})$  of 
$v=w_{j, a_{j}}w_{j+1, a_{j+1}}\cdots w_{r, a_{r}}$ 
obtained by construction 1 for the standard increasing ordering 
of $\{j_{q}, j_{q+1}, \cdots, j_{m}\}$, we have 
\begin{equation}\label{eqn:wqv} 
 w_{q}^{-1}(\alpha_{a_{j}})\geq v^{-1}(\alpha_{a_{j}}).
\end{equation} 

Again by Lemma \ref{lem:commute}, we see that $s_{a_{j}}$ 
commutes with $w_{i}$ for all $1\leq i \leq  q-1$. Therefore, we have 
\begin{equation}\label{eqn:wv} 
w^{-1}(\alpha_{a_{j}})=v^{-1}(\alpha_{a_{j}}).
\end{equation}

Now, the proof of the lemma follows from (\ref{eqn:w'w}), 
(\ref{eqn:wqv}) and (\ref{eqn:wv}).
\end{proof}

\subsection{Equality of Weyl groups in type $D_n$}
\label{subsec:weylD}

Let $G$ be of type $D_n$, where $n \geq 4$; then $G = \Spin_{2n}$. 
We order the simple roots as in \cite[p.~58]{HU}; then the minuscule 
simple roots are $\alpha_1$, $\alpha_{n-1}$ and $\alpha_n$.
We will obtain a slightly weaker version of Theorem \ref{thm:WMG} 
in this setting:

\begin{theorem}\label{thm:WMD}
We have
$W(G_{w_1,\hx}, T_{w_1}) = W(G_{w_1,w_1 x_1 }, T_{w_1})$.
\end{theorem}

\begin{remark}\label{rem:WMD}
The group $W(G_{w_1, w_1 x_1}, T_{w_1})$ is not 
necessarily generated by simple reflections. For example,  take $n=8$ and
$w = (s_4 s_5 s_6 s_8)(s_3 s_4 s_5 s_6 s_7)(s_1 s_2 s_3 s_4 s_5 s_6 s_8)$. 
Then, we have $w_{1}=(s_{4}s_{5}s_{6}s_{8})(s_{3}s_{4}s_{5}s_{6}s_{7})$
for the standard ordering of $\Peaks(Q_w)$. Let $\beta=\sum_{i=4}^7 \alpha_i$.
Then $v=s_{\beta} \in W(G_{w_1, w_1 x_1}, T_{w_1})$ but 
$v$ is not a product of simple reflections in this group.
%$W(G_{w_1})_{w_1 x_1}, T_{w_1})$.
\end{remark}

To show Theorem \ref{thm:WMD}, 
it suffices to consider the cases where $P = P^{\alpha_{1}}$, 
$P = P^{\alpha_{n}}$, since there is an automorphism of the Dynkin diagram 
of $G$ sending $\alpha_{n-1}$ to $\alpha_{n}$.

We begin with the easy case where $P = P^{\alpha_1}$. 
By arguing as in the beginning of the proof of Lemma \ref{lemWM}, 
it suffices to prove that 
$W(G_{w_1, w_1 x_1}, T_{w_1}) \subset W(G_{w_1, wx}, T_{w_1})$. 

If there is a unique peak of $Q_{w}$, then by construction 1, 
we have $w=w_{1}$ and so we are done. Otherwise, by the same construction, 
we have $w=s_{n}s_{n-1}s_{n-2}\cdots s_{1}$. 
% and $w_1 = w_n$ or $w_{n-1}$ ?
Since there is an automorphism 
of the Dynkin diagram of $G$ sending $\alpha_{n}$ to $\alpha_{n-1}$,
without loss of generality, we may assume that $w_{1}=s_{n}$ and 
$w_{2}=s_{n-1}s_{n-2}\cdots s_{1}$. Therefore, we have 
$W(G_{w_1}, T_{w_1}) = W(G_{w_1}, T_{w_1}) = \{1, s_{n}\}$. 
Thus, identity is the only element of $W(G_{w_1}, T_{w_1})$ 
that fixes $w_1 x_1$. So, we are done.

We now turn to the case where $P = P^{\alpha_n} = P_I$, where
$I = S \setminus \{ \alpha_n \}$.
For $1\leq i \leq n-1$, $1\leq a_{i}\leq n-2$, let 
\[ v_{i,a_{i}}=\begin{cases} 
s_{a_{i}}s_{a_{i}+1}\cdots s_{n-2}s_{n} & \mbox{if} 
~ i ~  \mbox{is odd}, \\ 
s_{a_{i}}s_{a_{i+1}}\cdots s_{n-2}s_{n-1} & \mbox{if} ~ i ~ \mbox{is even}.\\ 
\end{cases} \]

Let $v_{i,n}=s_{n}$ if $i$ is odd, $v_{i, n-1}=s_{n-1}$ if $i$ is even.

\begin{lemma}\label{lemWlong}
The minimal representative $w_0^I\in W^I$
of the longest element $w_0 \in W$ is of the form 
$w_0^I = \begin{cases} s_{n-1}(v_{n-2,n-2}v_{n-3,n-3}\cdots v_{1,1}) 
& \mbox{if} ~ n ~  \mbox{is odd}, \\ 
s_{n}(v_{n-2,n-2}v_{n-3,n-3}\cdots v_{1,1}) 
& \mbox{if} ~ n ~ \mbox{is even}.\\ \end{cases}$ 

Further, $(s_{n-1}, v_{n-2,n-2}, v_{n-3,n-3}, \ldots , v_{1,1})$ (respectively,  
$(s_{n}, v_{n-2,n-2}, v_{n-3,n-3}, \ldots ,  v_{1,1})$) is a generalised reduced 
decomposition of $w_0^I$ if $n$ is odd (respectively, if $n$ is even).
\end{lemma}

\begin{proof}
By induction on $n$ (=rank$(G)$), since there is an automorphism 
of the Dynkin diagram of $G$ sending $\alpha_{n-1}$ to $\alpha_{n}$, 
the minimal representative
\[ v \in W(G_{S\setminus \{\alpha_{1}\}}, 
G_{S\setminus \{\alpha_{1}\}}\cap T)^{S\setminus \{\alpha_{1}, \alpha_{n-1}\}} \] 
of the longest element in 
$W(G_{S\setminus \{\alpha_{1}\}}, G_{S\setminus \{\alpha_{1}\}}\cap T)$
is of the form 
\[ v = \begin{cases} 
s_{n-1}(v_{n-2,n-2}v_{n-3,n-3}\cdots v_{2,2}) & \mbox{if} ~ n ~  \mbox{is odd}, \\ 
s_{n}(v_{n-2,n-2}v_{n-3,n-3}\cdots v_{2,2}) 
& \mbox{if} ~ n ~ \mbox{is even}.\\ \end{cases} \] 

Again since there is an automorphism of the Dynkin diagram of $G$ 
sending $\alpha_{n-1}$ to $\alpha_{n}$, the number of positive roots of the form 
$\sum_{i\neq 1}m_{i}\alpha_{i}$, with $m_{n}\geq 1$ is equal to 
$\ell(v)$.
On the other hand, the number of positive roots 
$\sum_{i=1}^{n}m_{i}\alpha_{i}$, with $m_{n}\geq 1$ and $m_{1}\geq 1$ 
is equal to $\ell(v_{1,1})$.

Therefore, the number of positive roots of the form 
$\sum_{i}m_{i}\alpha_{i}$ with $m_{n}\geq 1$ is equal to 
${\ell}(v)+{\ell}(v_{1,1})$. 

Thus, we have  ${\ell}(w_0^I)={\ell}(v)+{\ell}(v_{1,1})$. 
This implies the assertion.
\end{proof}

Let $w\in W^I$ be such that $\Supp(w)=S$. By Lemma \ref{lemWlong} 
there exists a unique increasing sequence 
$1=a_{1} < a_{2} < \ldots < a_{k}\leq n$ of integers such that 
$w=v_{k,a_{k}}v_{k-1,a_{k-1}}\cdots v_{1,a_{1}}$. Further, again 
by Lemma \ref{lemWlong}, 
$(v_{k,a_{k}}, v_{k-1,a_{k-1}}, \ldots , v_{1,a_{1}})$ 
is a generalised reduced decomposition of $w$.

Let 
$J(w):=\{1\leq j \leq k-1: 
\langle \alpha_{a_{j+1}}, \alpha_{a_{j}}^{\vee} \rangle = 0\} \cup \{k\}$ 
and let $|J(w)|=m$. The peaks of the quiver $Q_{w}$ 
are indexed by the elements of $J(w)$. So, let $\{i_{j} : j\in J(w)\}$ 
be the peaks of $Q_{w}$. Then we have $\beta(i_{j})=\alpha_{a_{j}}$ 
for all $j\in J(w)$. Let $j_{1}=k > \cdots > j_{m}$ be the standard 
decreasing ordering of the elements of $J(w)$, that is; the ordering 
induced by the decreasing ordering of positive integers. Let 
$\hw = (w_1, w_2, \ldots , w_m)$ be the generalised reduced 
decomposition corresponding to this ordering of $\Peaks(Q_{w})$ 
obtained by construction 1. 

\begin{lemma}\label{lemcombD}
Let $1\leq l \leq k$ be a positive integer. 
\begin{enumerate}
\item 
Assume that $l$ is even. Then for any integer $1+a_{l}\leq  i \leq n-1$, 
we have $s_{i}v_{l,a_{l}}=v_{l,a_{l}}s_{i-1}$. \\
\item Assume that $l$ is odd. Then for any integer $1+a_{l}\leq  i \leq n-2$, 
we have $s_{i}v_{l,a_{l}}=v_{l,a_{l}}s_{i-1}$, and 
$s_{n}v_{l,a_{l}}=v_{l,a_{l}}s_{n-2}$. \\
\end{enumerate}
\end{lemma}

\begin{proof}
We consider the case where $l$ is even. The proof of the case where $l$ is odd 
is similar.

Since $s_{i}s_{j}=s_{j}s_{i}$ for all $1+a_{l} \leq j \leq i-2$, $s_{i}s_{i-1}s_{i}=s_{i-1}s_{i}s_{i-1}$ and $s_{i-1}s_{t}=s_{t}s_{i-1}$ for all 
$i+1\leq t\leq n-1$, we have $s_{i}v_{l, a_{l}}=v_{l,a_{l}}s_{i-1}$.

\end{proof}

Let $1\leq r \leq k$ be the least positive integer such that 
$\langle \alpha_{a_{j+1}} , \alpha_{a_{j}}^{\vee} \rangle \neq 0$ 
for all $r \leq j \leq k-1$. 
Then by construction 1, we have 
$w_{1}=v_{k,a_{k}}v_{k-1, a_{k-1}} \cdots v_{r,a_{r}}$.

We may now prove Theorem \ref{thm:WMD} for the generalised reduced 
decomposition associated to the above standard ordering.

By arguing as in the beginning of the proof of Lemma \ref{lemWM}, 
it suffices to prove that 
\[ W(G_{w_1, w_1 x_1}, T_{w_1}) \subset W(G_{w_1, wx}, T_{w_1}). \] 

Let $v\in W(G_{w_1}, T_{w_1})$  be such that 
$vw_{1}x_{1}=w_{1}x_{1}$. Then, we have $vw_{1}=w_{1}u$ for some 
$u\in W(G_{w_{1}}, T_{w_{1}})$ such that $s_{n}\nleq u$ if $r$ is odd 
and $s_{n-1}\nleq u$ if $r$ is even.

We prove by induction on $\ell(u)$ that 
\[ uv_{r-1,a_{r-1}}v_{r-2,a_{r-2}}\cdots v_{1,a_{1}}
= v_{r-1,a_{r-1}}v_{r-2,a_{r-2}}\cdots v_{1,a_{1}} ~ {\rm mod} ~ W^{\alpha_{n}}. \]

If $\ell(u)=1$, then we have $u=s_{t}$ for some integer $a_{r}\leq t \leq n$ and 
$t\neq n$ if $r$ is odd and $t\neq n-1$ if $r$ is even. Using Lemma 
\ref{lemcombD} repeatedly, for any such $t$, we see that  
$s_{t}v_{r-1,a_{r-1}}v_{r-2,a_{r-2}}\cdots v_{1,a_{1}}
=v_{r-1,a_{r-1}}v_{r-2,a_{r-2}}\cdots v_{1,a_{1}}s_{l}$ 
for some integer $l\neq n$.

So, assume that $\ell(u)\geq 2$. Choose an integer $a_{r}\leq t \leq n$  
such that $t\neq n$ if $r$ is odd and $t\neq n-1$ if $r$ is even such that 
$\ell(us_{t})=\ell(u)-1$. Since $\ell(us_{t})=\ell(u)-1$, by induction 
on $\ell(u)$, we see that 
$us_{t}v_{r-1,a_{r-1}}v_{r-2,a_{r-2}} \cdots v_{1,a_{1}}
= v_{r-1,a_{r-1}}v_{r-2,a_{r-2}}\cdots v_{1,a_{1}} ~{\rm mod}~ W^{\alpha_{n}}$. 
By the above discussion, we have 
$s_{t}v_{r-1,a_{r-1}}v_{r-2,a_{r-2}} \cdots v_{1,a_{1}}
= v_{r-1,a_{r-1}}v_{r-2,a_{r-2}} \cdots v_{1,a_{1}}s_{l}$ 
for some integer $l\neq n$.

Hence, we have 
$uv_{r-1,a_{r-1}}v_{r-2,a_{r-2}}\cdots v_{1,a_{1}}
=v_{r-1,a_{r-1}}v_{r-2,a_{r-2}}\cdots v_{1,a_{1}} ~{\rm mod}~ W^{\alpha_{n}}$. 
Therefore, we have 
$vw=w_{1}uv_{r-1,a_{r-1}}v_{r-2,a_{r-2}}\cdots v_{1,a_{1}}
= w ~{\rm mod}~ W^{\alpha_{n}}$. 
Thus, we obtain $vwx=wx$. This completes the proof for the above (special) 
generalised reduced decomposition.

Next, we turn to the general case. Recall that 
\[ J(w) =
\{1\leq j \leq k-1: \langle \alpha_{a_{j+1}}, \alpha_{a_{j}}^{\vee} \rangle = 0\}
\cup \{k\} \] 
and let $|J(w)|=m$.
Also,  let $\preceq$ be an arbitrary ordering of the elements of $J(w)$ and let 
$j\in J(w)$ be the first element in this ordering. Since $\Peaks(Q_{w})$ 
is indexed by $J(w)$, this ordering of $J(w)$ induces an ordering of 
$\Peaks(Q_{w})$. 

Let $\hw' = (w'_1, w'_2, \ldots , w'_m)$ be the generalised reduced 
decomposition of $w$ obtained by construction 1 corresponding 
to this ordering of $\Peaks(Q_{w})$. Let $1\leq q \leq m$ be the integer 
such that $j=j_{q}$. Again by construction 1, we have 
$w_{q}=w'_{1}v$ for some $v\in W$ such that 
$\ell(w_q)=\ell(w'_1)+\ell(v)$.

Let $\hX(\hw')$ be the variety corresponding to this ordering. Recall that 
$\hX(\hw)$ denotes the variety corresponding to the standard decreasing 
ordering of the elements of $J(w)$.

Recall from Lemma \ref{lemMT} that 
%we have $(G_{w'_{1}})_{\hx}\subset (G_{w'_{1}})_{w'_{1}x'_{1}}$ 
%and $T_{w'_{1}}$ is a common maximal torus of both  
%$(G_{w'_{1}})_{\hx'}$ and $(G_{w'_{1}})_{w'_{1}x'_{1}}$. Further, 
$W(G_{w'_1, \hx'}, T_{w'_1}) \subset W(G_{w'_1, w'_1 x'_1}, T_{w'_1}) \subset W$.

\begin{proposition}\label{propWMGD} 
Assume that $q\neq 1$. That is, $w_{1}\neq w'_{1}$. Then
\[ W(G_{w'_1, \hx'}, T_{w'_1}) = W(G_{w'_1, w'_1 x'_1}, T_{w'_1}) \] 
and this group is generated by simple reflections.
\end{proposition}

\begin{proof}
As at the beginning of the proof of Lemma \ref{lemWM}, it suffices to prove 
the inclusion 
\[ W(G_{w'_1, w'_1 x'_1}, T_{w'_1 }) \subset W(G_{w'_1, wx}, T_{w'_1}). \] 
Recall the definition of $Q_w(A)$ for any $A \subset \Peaks(Q_w)$ from
Definition \ref{def:subquiver}. 
By construction 1, for any vertex $l$ of $Q_{w}(\Peaks(Q_{w})\setminus \{i_j\})$ 
and for any vertex of $i$ of $Q_{w}(\{i_j\})$, we have $(l,i)\notin R$. 
In particular, we have $\langle \beta(l)^{\vee} , \beta(i) \rangle =0$. 
Thus, every simple reflection $s_{e}\leq w'_{1}$ commutes with every simple 
reflection $s_{f}\leq w_r$ for all $1 \leq r \leq q -1$. 
Hence, $w'_{1}$ commutes with $w_{1}$. 
Further, either $s_{n}\leq w_{1}$ or $s_{n-1}\leq w_{1}$. If $w'_{1}=s_{n-1}$ 
or $w'_{1}=s_{n}$ we are done. Otherwise, we have $s_{l}\nleq w'_{1}$ 
for all $n-2\leq l \leq n$. Therefore the Dynkin diagram of $G_{w'_1}$ 
is of type $A$. 

Let $h=j_{q+1}$. By construction 1, we have $a_{s+1}=a_{s}+1$ for all 
$h+1\leq s \leq j-1$. Again, by construction 1, there exists 
a sequence $a_{s}\leq b_{s} \leq n-3$  ($h+1\leq s \leq j$) of positive integers 
with $b_{s+1}=1+b_{s}$ for all $h+1 \leq s \leq j-1$ such that 
$(w'_{1})^{-1} = w_{a_{h+1}, b_{h+1}} w_{a_{h+2}, b_{h+2}}\cdots w_{a_{j}, b_{j}}$.

Therefore, by Lemma \ref{lemcA}, for any integer $a_{h+1}\leq t \leq b_{j}$ 
different from $b_{h+1}$, we have $s_{t}(w'_{1})^{-1}=(w'_{1})^{-1}s_{p}$ 
for some integer $p\neq a_{j}$ such that $s_{p}\leq w'$.

Hence we have 

{\it Observation}. For any $a_{h+1}\leq t \leq b_{j}$ different from $b_{h+1}$, 
we have $w'_{1}s_{t}=s_{p}w'_{1}$ for some integer $p\neq a_{j}$ such that 
$s_{p}\leq w'$.

Therefore, we can imitate the proof of Lemma \ref{lemWM}.

{\it Claim}. $W(G_{w'_1, w'_1 x'_1}, T_{w'_1})$ 
is generated by simple reflections.

Let $v\in W(G_{w'_1}, T_{w'_1})$ be such that 
$vw'_{1}=w'_{1}u$ for some $u\in W(G_{w'_1}, T_{w'_1})$. 
We prove by induction on ${\ell}(u)$ that $v$ is a product of simple reflections 
$s_{p}\leq w'_{1}$ such that $s_{p}w'_{1}=w'_{1}s_{t}$ for some integer 
$t\neq b_{h+1}$ with  $s_{t}\leq w'_{1}$. 

If ${\ell}(u)=1$, then we have $u=s_{t}$ for some integer $t\neq b_{h+1}$ 
such that $s_{t}\leq w'_{1}$. Hence by Observation, we have $v=s_{p}$ 
for some integer $p\neq a_{j}$ such that $s_{p}\leq w'_{1}$. 

So, assume that ${\ell}(u)\geq 2$. Choose an integer $t\neq b_{h+1}$ 
such that $s_{t}\leq w'_{1}$ and ${\ell}(us_{t})={\ell}(u)-1$. 

We have $vs_{w'_{1}(\alpha_{t})}w'_{1}=vw'_{1}s_{t}=w'_{1}us_{t}$. 
Note that $w'_{1}(\alpha_{t})\in R^{+}\setminus  v^{-1}(R^{+})$. 
Therefore $s_{w'_{1}(\alpha_{t})}\in W(G_{w'_{1}}, T_{w'_{1}})$. 
Since ${\ell}(us_{t})={\ell}(u)-1$, by induction $vs_{w'_{1}(\alpha_{t})}$ 
is a product of simple reflections $s_{p}\leq w'_{1}$ such that 
$s_{p}w'_{1}=w'_{1}s_{l}$ for some integer $l\neq b_{h+1}$ such that 
$s_{l}\leq w'_{1}$.

On the other hand, by Observation, we have $w'_{1}s_{t}=s_{e}w'_{1}$ 
for some integer $e\neq a_{j}$ such that $s_{e}\leq w'_{1}$. Therefore, 
we have  $w'_{1}(\alpha_{t})=\alpha_{e}$ and so $v$ is a product 
of simple reflections $s_{p}\leq w'_{1}$ such that $s_{p}w'_{1}=w'_{1}s_{l}$
for some integer $l\neq b_{h+1}$ such that $s_{l}\leq w'_{1}$.

This proves the claim.

Now, let $v\in W(G_{w'_1,w'_1 x'_1}, T_{w'_1})$. 
Then by the claim, $v$ is a product of simple reflections $s_{p}\leq w'_{1}$ 
such that $s_{p}w'_{1}=w'_{1}s_{l}$ for some integer $l\neq b_{h+1}$ 
such that $s_{l}\leq w'_{1}$. On the other hand,  by Lemma \ref{lemsr}, 
for any such $s_{p}$, we have $s_{p}wx=wx$ and hence we have $vwx=wx$.
\end{proof}

\subsection{Root inequality in type $D_n$}
\label{subsec:rootD} 

We keep the notation of Subsection \ref{subsec:weylD}
In particular, $w\in W$ denotes a minuscule element, 
$P:= P^w = P^{\alpha_r}$ for some $r=1, n-1, n$,
and $\preceq$ an arbitrary ordering of the elements of $J(w)$.
Let $\hw' = (w'_1, w'_2, \ldots , w'_m)$ be the generalised reduced 
decomposition of $w$ obtained by construction 1 for this ordering.
Then we have 

\begin{proposition}\label{prop:heightD}
Let $\alpha$ be the unique simple root such that $(w'_{1})^{-1}(\alpha)$ 
is a negative root (see \cite[Lem.~5.1]{BK}). Then we have 
$(w'_{1})^{-1}(\alpha) \geq  w^{-1}(\alpha)$. 
Further, $(w'_{1})^{-1}(\alpha) = w^{-1}(\alpha)$ if and only if $w = w'_{1}$; 
that is, $m=1$.
\end{proposition}

\begin{proof}
It suffices to consider the cases where $P = P^{\alpha_{1}}$, 
$P = P^{\alpha_{n}}$, since there is an automorphism of the Dynkin 
diagram of $G$ sending $\alpha_{n-1}$ to $\alpha_{n}$.

Also, the case where $P = P^{\alpha_1}$ is easy: 
if there is a unique peak of $Q_{w}$, then by construction 1, 
we have $w=w_{1}$ and so we are done. Otherwise, by the same construction, 
we have $w=s_{n}s_{n-1}s_{n-2}\cdots s_{1}$. 
Using again the existence of an automorphism of the Dynkin diagram 
sending $\alpha_{n}$ to $\alpha_{n-1}$, without loss of generality, 
we may assume that $w_{1}=s_{n}$ and $w_{2}=s_{n-1}s_{n-2}\cdots s_{1}$. 
Therefore, we have $\alpha=\alpha_{n}$ and 
$w_{1}^{-1}(\alpha_{n})=-\alpha_{n} 
> -(\alpha_{n}+\sum_{l=1}^{n-2}\alpha_{l})=w^{-1}(\alpha_{n})$. 
Hence, we are done.

Now, we turn to the case $P=P^{\alpha_{n}}$. We adapt the argument 
of the proof of Proposition \ref{prop:heightA}.

Recall that by Lemma \ref{lemWlong} there exists a unique increasing sequence 
$1=a_{1} < a_{2} < \ldots < a_{k}\leq n$ of integers such that 
$w=v_{k,a_{k}}v_{k-1,a_{k-1}}\cdots v_{1,a_{1}}$.

{\it Claim.} The proposition holds for the standard decreasing ordering 
$j_{1}=k > j_{2} > \cdots > j_{m}$ of the elements of $J(w)$.

Here again, we may assume that $m\geq 2$. Let $r=j_{2}+1$.
We consider the case where $k$ is even. The proof for the case 
where $k$ is odd is similar. By construction 1, we have 
$w_{1}=v_{k,a_{k}}v_{k-1, a_{k-1}} \cdots v_{r,a_{r}}$. 
Further, we have $\alpha=\alpha_{a_{k}}$. Therefore, we have 
\[ w_{1}^{-1}(\alpha)=
\begin{cases}-(\sum_{l=a_{r}}^{n-1}\alpha_{l}) ~ \mbox{if} ~  r=k\\
-(\sum_{l=a_{r}}^{n}\alpha_{l}) ~ \mbox{if} ~ r=k-1 \\
-(\sum_{l=a_{r}}^{n-(k+1-r)}\alpha_{l}
+2(\sum_{l=n-(k-r)}^{n-2}\alpha_{l})+\alpha_{n-1}+\alpha_{n}) 
~ \mbox{if} ~ r\leq k-2.\end{cases}.\\ \]

{\it Case 1.} $~ r=k$. 

{\it Subcase 1.} $~ r=2$. We have 
$w^{-1}(\alpha)=-(\sum_{l=a_{2}-1}^{n}\alpha_{l})$.

{\it Subcase 2.} $~ r\geq 3$. %Since $r=k$ is even, we have $r\geq 4$.
We have 
\[ w^{-1}(\alpha)=-(\sum_{l=a_{k}-k+1}^{n-k}
\alpha_{l}+2(\sum_{l=n-k+1}^{n-2}\alpha_{l})+\alpha_{n-1}+\alpha_{n}).\] 

{\it Case 2.}  $~ r=k-1$. We have 
\[ w^{-1}(\alpha)=-(\sum_{l=a_{k}-k+1}^{n-k}
\alpha_{l}+2(\sum_{l=n-k+1}^{n-2}\alpha_{l})+\alpha_{n-1}+\alpha_{n}).\]

{\it Case 3.} $~ r\leq k-2$. We have 
\[ w^{-1}(\alpha)=-(\sum_{l=a_{k}-k+1}^{n-k}
\alpha_{l}+2(\sum_{l=n-k+1}^{n-2}\alpha_{l})+\alpha_{n-1}+\alpha_{n}).\]

This completes the proof of the claim.

We now prove the proposition for an arbitrary ordering of the elements 
of $J(w)$.

Let $j\in J(w)$ be the first element in this ordering. Let $1\leq q \leq m$ 
be the integer such that $j=j_{q}$. By the proof for the standard decreasing 
ordering, we may assume that $q\neq 1$. By construction 1, we have 
$w_{q}=w'_{1}v$ for some $v\in W$ 
such that ${\ell}(w_{q})={\ell}(w'_{1})+{\ell}(v)$.

Let $h=j_{q+1}$. Then by construction 1, there exists a sequence 
$a_{l} \leq b_{l} \leq n$ ($h+1 \leq l \leq j$) of integers such that 
$b_{l+1}=b_{l}+1$ for all $h+1\leq l \leq j-1$, and we have
$(w'_{1})^{-1}=w_{a_{h+1},b_{h+1}}w_{a_{h+2},b_{h+2}}\cdots w_{a_{j},b_{j}}$. 
Further, we have $a_{l+1}=a_{l}+1$ for all $h+1\leq l \leq j-1$. 
By Lemma \ref{lem:commute}, 
$w'_{1}$ commutes with $w_{l}$ for all $1\leq l \leq q-1$.
Therefore, we have $b_j \leq n-3$.

Now, we have $\alpha=\alpha_{a_{j}}$. Therefore, 
\[ (w'_{1})^{-1}(\alpha_{a_{j}})= 
-\sum_{i=a_{h+1}}^{b_{j}}\alpha_{i}. \]  

On the other hand, we have 
\[
w_{q}^{-1}(\alpha)=\begin{cases}
-(\sum_{l=a_{h+1}}^{n-1}\alpha_{l}) 
~ \mbox{if} ~  j=h+1 ~ \mbox{and ~ is ~ even }\\
-((\sum_{l=a_{h+1}}^{n-2}\alpha_{l})+\alpha_{n}) 
~ \mbox{if} ~  j=h+1 ~ \mbox{and ~ is ~ odd }\\
-(\sum_{l=a_{h+1}}^{n}\alpha_{l}) 
~ \mbox{if} ~ h=j-2 \\
-(\sum_{l=a_{h+1}}^{n+h-j}\alpha_{l}
+2(\sum_{l=n+h+1-j)}^{n-2}\alpha_{l})+\alpha_{n-1}+\alpha_{n}) 
~ \mbox{if} ~ h\leq j-3.\end{cases}.\\ \] 

Since $b_{j}\leq n-3$, we have 
\begin{equation}\label{eqn:w'wD}
(w'_{1})^{-1}(\alpha_{a_{j}}) > 
w_{q}^{-1}(\alpha_{a_{j}}). 
\end{equation} 

Applying the claim to the generalised reduced decomposition 
$\hat{v}=(w_{q}, w_{q+1}, \cdots , w_{m})$  of 
$v=w_{j, a_{j}}w_{j+1, a_{j+1}}\cdots w_{r, a_{r}}$ 
obtained by construction 1 corresponding
to the standard increasing ordering of $\{j_{q}, j_{q+1}, \cdots, j_{m}\}$, 
we obtain 
\begin{equation}\label{eqn:wqvD} 
w_{q}^{-1}(\alpha_{a_{j}})\geq v^{-1}(\alpha_{a_{j}}).
\end{equation} 

Again by Lemma \ref{lem:commute}, we see that $s_{a_{j}}$ 
commutes with $w_{i}$ for all $1\leq i \leq  q-1$. Therefore, we have 
\begin{equation}\label{eqn:wvD} 
w^{-1}(\alpha_{a_{j}})=v^{-1}(\alpha_{a_{j}}).
\end{equation}

Now, the proof of the lemma follows from (\ref{eqn:w'wD}), (\ref{eqn:wqvD}) 
and (\ref{eqn:wvD}).
\end{proof}

\subsection{Type $E_6$}
\label{subsec:E6}

Let $G$ be of type $E_6$. We order the simple roots as in 
\cite[p.58]{HU}; then the minuscule simple roots are $\alpha_1$ 
and $\alpha_6$. Since there is an automorphism of the Dynkin 
diagram of $G$ taking $\alpha_1$ to $\alpha_6$, it is sufficient 
to consider Schubert varieties in $G/P$, where
$P =  P^{\alpha_1} = P_I$ with $I = S \setminus \{ \alpha_1 \}$. 

A reduced decomposition of the longest element $w_0^I \in W^I$ is
\[ (s_6, s_5, s_4, s_3, s_2, s_4, s_1, s_3, s_5, s_4, s_6, s_5, s_2, s_4, s_3, s_1) \]
and all the $w \in W^I$ are obtained by taking certain reduced subexpressions 
of the above one. For any such $w$, we define the standard ordering on 
$\Peaks(Q_w)$ as the ordering induced by the standard increasing ordering 
on its vertices (viewed as positive integers). Using the smoothness
criterion of \cite[Thm.~7.11]{NP}, one may check that 
the $w \in W^I$ such that $\Supp(w)=S$ and $\hX(\hw)$ is smooth 
for this standard ordering are exactly the following:
\[ s_{6}s_{5}s_{2}s_{4}s_{3}s_{1}, 
\quad s_{4}s_{6}s_{5}s_{2}s_{4}s_{3}s_{1}, 
\quad s_{3}s_{4}s_{6}s_{5}s_{2}s_{4}s_{3}s_{1},
\quad s_{1}s_{3}s_{4}s_{6}s_{5}s_{2}s_{4}s_{3}s_{1}, \] 
\[ s_{4}s_{1}s_{3}s_{5}s_{4}s_{6}s_{5}s_{2}s_{4}s_{3}s_{1},\quad s_{2}s_{4}s_{1}s_{3}s_{5}s_{4}s_{6}s_{5}s_{2}s_{4}s_{3}s_{1}, 
\quad s_{6}s_{5}s_{4}s_{3}s_{2}s_{4}s_{1}s_{3}s_{5}s_{4}s_{6}s_{5}
s_{2}s_{4}s_{3}s_{1} = w_0^I. \]
We now describe the varieties associated to the generalised reduced 
decompositions of these elements obtained by construction 1.

If $w = w_0^I$, then $X(w) = G/P$. Thus, there is a unique peak 
and $\hX(\hw) = X(w)$.

In all other cases, there are two peaks and hence two decompositions, 
$\hw$ (for the standard ordering) and $\hw'$ (for the nonstandard one).

For $w = s_{6}s_{5}s_{2}s_{4}s_{3}s_{1}$, we have 
$\hw =(s_{6}s_{5} , s_{2}s_{4}s_{3}s_{1})$. Thus, 
$\hf : \hX(\hw) \longrightarrow X(s_{6}s_{5}) \simeq \bP^{2}$ is a fibration
with fiber $X(s_{2}s_{4}s_{3}s_{1}) \simeq \bP^{4}$. Moreover,
$\hw'=(s_{2}, s_{6}s_{5}s_{4}s_{3}s_{1})$. So, the morphism 
$\hf' : \hX(\hw') \longrightarrow X(s_{2}) \simeq \bP^{1}$ is a fibration 
with fiber $X(s_{6}s_{5}s_{4}s_{3}s_{1}) \simeq \bP^{5}$.

For $w = s_{4}s_{6}s_{5}s_{2}s_{4}s_{3}s_{1}$, we have  
$\hw =(s_{4}s_{2} , s_{6}s_{5}s_{4}s_{3}s_{1})$.
Thus, 
$\hf : \hX(\hw) \longrightarrow X(s_{4}s_{2}) \simeq \bP^{2}$ 
is a fibration with fiber $X(s_{6}s_{5}s_{4}s_{3}s_{1}) \simeq \bP^{5}$. 
Also, $\hw' = (s_{6} ,s_{4}s_{5}s_{2}s_{4}s_{3}s_{1})$. 
The  $\alpha_{4}$ is not minuscule in the Dynkin diagram of 
$\Supp(s_{4}s_{5}s_{2}s_{4}s_{3}s_{1})$. Hence $\hX(\hw')$ is 
singular in view of \cite[Thm.~7.11]{NP} again.

For $w = s_{3}s_{4}s_{6}s_{5}s_{2}s_{4}s_{3}s_{1}$, we have  
$\hw = (s_{3}s_{4}s_{2} , s_{6}s_{5}s_{4}s_{3}s_{1})$. 
Thus, the morphism
$\hf : \hX(\hw) \longrightarrow X(s_{3}s_{4}s_{2}) \simeq \bP^{3}$ 
is a fibration with fiber $X(s_{6}s_{5}s_{4}s_{3}s_{1}) \simeq \bP^{5}$.
Also, $\hw' = (s_{6},s_{3}s_{4}s_{5}s_{2}s_{4}s_{3}s_{1})$. 
The simple root $\alpha_{3}$ is not minuscule in the Dynkin diagram of 
$\Supp(s_{3}s_{4}s_{5}s_{2}s_{4}s_{3}s_{1})$. Hence $\hX(\hw')$ is singular.

For $w = s_{1}s_{3}s_{4}s_{6}s_{5}s_{2}s_{4}s_{3}s_{1}$, we have  
$\hw =(s_{1}s_{3}s_{4}s_{2} , s_{6}s_{5}s_{4}s_{3}s_{1})$. 
Thus, the morphism
$\hf : \hX(\hw) \longrightarrow X(s_{1}s_{3}s_{4}s_{2}) \simeq \bP^{4}$ 
is a fibration with fiber $X(s_{6}s_{5}s_{4}s_{3}s_{1}) \simeq \bP^{5}$. 
Also, $\hw' = (s_{6} , s_{1}s_{3}s_{4}s_{5}s_{2}s_{4}s_{3}s_{1})$. 
The simple root $\alpha_{1}$ is minuscule in the Dynkin diagram of 
$\Supp(s_{1}s_{3}s_{4}s_{5}s_{2}s_{4}s_{3}s_{1})$. Hence $\hX(\hw')$ 
is smooth. The morphism 
$\hf':  \hX(\hw') \longrightarrow X(s_{6}) \simeq \bP^{1}$ is a fibration 
with fiber $G_v/G_v \bigcap P$, where 
$v=s_{1}s_{3}s_{4}s_{5}s_{2}s_{4}s_{3}s_{1}$. Moreover, $G_v$ is of
type $D_5$ and $G_v/G_v \cap P$ is isomorphic to the quadric $\bQ^8$.

For $w = s_{4}s_{1}s_{3}s_{5}s_{4}s_{6}s_{5}s_{2}s_{4}s_{3}s_{1}$, 
we have  
$\hw = (s_{4}s_{5}s_{6}, s_{1}s_{3}s_{4}s_{5}s_{2}s_{4}s_{3}s_{1})
= (s_4 s_5 s_6, v)$, where $v$ is as above. So, the morphism 
$\hf : \hX(\hw) \longrightarrow X(s_{4}s_{5}s_{6}) \simeq \bP^{3}$ 
is a fibration with fiber $\bQ^8$. Also,
$\hw'=(s_{1}, s_{4}s_{5}s_{6}s_{3}s_{4}s_{5}s_{2}s_{4}s_{3}s_{1})$. 
The simple root $\alpha_{4}$ is not minuscule in the Dynkin diagram of 
$\Supp(s_{4}s_{5}s_{6}s_{3}s_{4}s_{5}s_{2}s_{4}s_{3}s_{1})$. Hence 
$\hX(\hw')$ is singular.

For $w = s_{2}s_{4}s_{1}s_{3}s_{5}s_{4}s_{6}s_{5}s_{2}s_{4}s_{3}s_{1}$, 
we have  
$\hw =(s_{2}s_{4}s_{5}s_{6},v)$. 
Thus, the morphism 
$\hf : \hX(\hw) \longrightarrow X(s_{2}s_{4}s_{5}s_{6}) \simeq \bP^{4}$ 
is a fibration with fiber $\bQ^8$ again. Also,  
$\hw'=(s_{1}, s_{2}s_{4}s_{5}s_{6}s_{3}s_{4}s_{5}s_{2}s_{4}s_{3}s_{1})$. 
The simple root $\alpha_{2}$ is not minuscule in the Dynkin diagram of 
$\Supp(s_{4}s_{5}s_{6}s_{3}s_{4}s_{5}s_{2}s_{4}s_{3}s_{1})$. 
Hence $\hX(\hw')$ is singular in this case, too.

Finally, there is a unique element $w\in W^I$ for which $\hX(\hw)$ is 
singular but $\hX(\hw')$ is smooth. Take 
$w=s_{1}s_{3}s_{5}s_{4}s_{6}s_{5}s_{2}s_{4}s_{3}s_{1}$. 
Then 
$\hw = (s_{1}s_{3}, s_{5}s_{4}s_{6}s_{5}s_{2}s_{4}s_{3}s_{1})$. 
The simple root $\alpha_{5}$ is not minuscule in the Dynkin diagram of 
$\Supp(s_{5}s_{4}s_{6}s_{5}s_{2}s_{4}s_{3}s_{1})$. 
Hence $\hX(\hw)$ is singular. Also,
$\hw' = (s_{5}s_{6}, v)$. Hence $\hX(\hw')$ is smooth. The morphism 
$\hf' : \hX(\hw') \longrightarrow X(s_{5}s_{6}) \simeq \bP^{2}$ 
is a fibration with fiber $\bQ^8$ once more.

\subsection{Type $E_7$}
\label{subsec:E7}

Let $G$ be of type $E_7$. Here $\varpi_7$ is the unique minuscule 
fundamental weight. Let $P = P^{\alpha_7} = P_I$, where
$I = S \setminus \{ \alpha_7 \}$. Then $w_0^I$ admits the reduced 
decomposition 
\[ (s_7, s_6, s_5, s_4, s_2, s_3, s_1, s_4, s_5, s_3, s_4, s_6, s_5, 
s_2, s_4, s_3, s_7, s_6, s_5, s_4, s_1, s_3, s_2, s_4, s_5, s_6, s_7). \] 
Like in type $E_6$, we obtain all the $w \in W^I$ by taking certain
reduced subexpressions of the above one, and we define the standard
ordering on the peaks of the associated quivers as the ordering induced
by the standard increasing order on vertices. 

Using again the smoothness criterion of \cite[Thm.~7.11]{NP},
one may check that the $w\in W^I$ such that $\Supp(w)=S$ and 
$\hX(\hw)$ is smooth for the standard ordering of $\Peaks(Q_{w})$ 
are exactly the following:

\[ s_{1}s_{3}s_{2}s_{4}s_{5}s_{6}s_{7}, \quad
s_{4}s_{1}s_{3}s_{2}s_{4}s_{5}s_{6}s_{7}, \quad
s_{5}s_{4}s_{1}s_{3}s_{2}s_{4}s_{5}s_{6}s_{7}, \quad
s_{6}s_{5}s_{4}s_{1}s_{3}s_{2}s_{4}s_{5}s_{6}s_{7}, \]
\[ s_{7}s_{6}s_{5}s_{4}s_{1}s_{3}s_{2}s_{4}s_{5}s_{6}s_{7}, \quad
s_{3}s_{7}s_{6}s_{5}s_{4}s_{1}s_{3}s_{2}s_{4}s_{5}s_{6}s_{7}, \quad 
s_{4}s_{3}s_{7}s_{6}s_{5}s_{4}s_{1}s_{3}s_{2}s_{4}s_{5}s_{6}s_{7}, \]
\[ s_{2}s_{4}s_{3}s_{7}s_{6}s_{5}s_{4}s_{1}s_{3}s_{2}s_{4}s_{5}s_{6}s_{7}, 
s_{5}s_{4}s_{3}s_{7}s_{6}s_{5}s_{4}s_{1}s_{3}s_{2}s_{4}s_{5}s_{6}s_{7},
s_{5}s_{2}s_{4}s_{3}s_{7}s_{6}s_{5}s_{4}s_{1}s_{3}s_{2}s_{4}s_{5}s_{6}s_{7}, \]
\[ s_{7}s_{6}s_{5}s_{4}s_{2}s_{3}s_{1}s_{4}s_{5}s_{3}s_{4}s_{6}s_{5}
s_{2}s_{4}s_{3}s_{7}s_{6}s_{5}s_{4}s_{1}s_{3}s_{2}s_{4}s_{5}s_{6}s_{7} = w_0^I. \]

We now describe the varieties associated to all the generalised reduced
decompositions of these elements obtained by construction 1.

Note that $w = w_0^I$ has a unique peak, and $X(w) = G/P = \hX(\hw)$. 
All other Weyl group elements except $s_{5}s_{2}s_{4}s_{3}s_{7}s_{6}s_{5}s_{4}s_{1}s_{3}s_{2}s_{4}s_{5}s_{6}s_{7}$
have two peaks, and hence two decompositions, $\hw$ (for the standard ordering)
and $\hw'$ (for the nonstandard one).

For $w = s_{1}s_{3}s_{2}s_{4}s_{5}s_{6}s_{7}$, we have  
$\hw = (s_{1}s_{3}, s_{2}s_{4}s_{5}s_{6}s_{7})$. 
So, the morphism 
$\hf: \hX(\hw) \longrightarrow X(s_{1}s_{3}) \simeq \bP^{2}$ 
is a fibration with fiber $X(s_{2}s_{4}s_{5}s_{6}s_{7}) \simeq \bP^{5}$. 
Also, $\hw' = (s_{2}, s_{1}s_{3}s_{4}s_{5}s_{6}s_{7})$. 
So, the morphism 
$\hf' :\hX(\hw') \longrightarrow X(s_{2}) \simeq \bP^{1}$ 
is a fibration with fiber 
$X(s_{1}s_{3}s_{4}s_{5}s_{6}s_{7}) \simeq \bP^{6}$.

For $w = s_{4}s_{1}s_{3}s_{2}s_{4}s_{5}s_{6}s_{7}$,  we have  
$\hw =(s_{4}s_{2}, s_{1}s_{3}s_{4}s_{5}s_{6}s_{7})$. 
Therefore, the morphism 
$\hf : \hX(\hw) \longrightarrow X(s_{4}s_{2}) \simeq \bP^{2}$ is a 
fibration with fiber $X(s_{1}s_{3}s_{4}s_{5}s_{6}s_{7}) \simeq \bP^{6}$. 
Also, $\hw' = (s_{1}, s_{4}s_{3}s_{2}s_{4}s_{5}s_{6}s_{7})$. 
The simple root $\alpha_{4}$ is not a minuscule root of the 
Dynkin diagram of $\Supp(s_{4}s_{3}s_{2}s_{4}s_{5}s_{6}s_{7})$. 
Therefore, $\hX(\hw')$ is singular.

For $w = s_{5}s_{4}s_{1}s_{3}s_{2}s_{4}s_{5}s_{6}s_{7}$, we have  
$\hw = (s_{5}s_{4}s_{2}, s_{1}s_{3}s_{4}s_{5}s_{6}s_{7})$. 
Thus, the morphism 
$\hf : \hX(\hw) \longrightarrow X(s_{5}s_{4}s_{2}) \simeq \bP^{3}$ 
is a fibration with fiber 
$X(s_{1}s_{3}s_{4}s_{5}s_{6}s_{7}) \simeq \bP^{6}$. 
Also, 
$\hw' = (s_{1}, s_{5}s_{4}s_{3}s_{2}s_{4}s_{5}s_{6}s_{7})$. 
Since $\alpha_{5}$ is not a minuscule root of the Dynkin diagram of 
$\Supp(s_{5}s_{4}s_{3}s_{2}s_{4}s_{5}s_{6}s_{7})$, we see that
$\hX(\hw')$ is singular.

For $w = s_{6}s_{5}s_{4}s_{1}s_{3}s_{2}s_{4}s_{5}s_{6}s_{7}$, we have  
$\hw = (s_{6}s_{5}s_{4}s_{2}, s_{1}s_{3}s_{4}s_{5}s_{6}s_{7})$. 
So, the morphism 
$\hf :\hX(\hw) \longrightarrow X(s_{6}s_{5}s_{4}s_{2}) \simeq \bP^{4}$ 
is a fibration with fiber $X(s_{1}s_{3}s_{4}s_{5}s_{6}s_{7}) \simeq \bP^{6}$.
Also, $\hw' = (s_{1}, s_{6}s_{5}s_{4}s_{3}s_{2}s_{4}s_{5}s_{6}s_{7})$. 
The simple root $\alpha_{6}$ is not a minuscule root of the Dynkin diagram of 
$\Supp(s_{6}s_{5}s_{4}s_{3}s_{2}s_{4}s_{5}s_{6}s_{7})$. Therefore, 
$\hX(\hw')$ is singular.

For $w = s_{7}s_{6}s_{5}s_{4}s_{1}s_{3}s_{2}s_{4}s_{5}s_{6}s_{7}$, we have  
$\hw = (s_{7}s_{6}s_{5}s_{4}s_{2}, s_{1}s_{3}s_{4}s_{5}s_{6}s_{7})$. 
So, the morphism 
$\hf : \hX(\hw) \longrightarrow X(s_{7}s_{6}s_{5}s_{4}s_{2}) \simeq \bP^{5}$ 
is a fibration with fiber $X(s_{1}s_{3}s_{4}s_{5}s_{6}s_{7}) \simeq \bP^{6}$. 
Also, 
$\hw'=(s_{1}, s_{7}s_{6}s_{5}s_{4}s_{3}s_{2}s_{4}s_{5}s_{6}s_{7})$.
So, the morphism 
$\hf': \hX(\hw') \longrightarrow X(s_{1}) \simeq \bP^{1}$ is a fibration 
with fiber $G_{v}/G_{v} \bigcap P$, where 
$v = s_{7}s_{6}s_{5}s_{4}s_{3}s_{2}s_{4}s_{5}s_{6}s_{7}$. Moreover,
$G_v$ is of type $D_6$ and $G_v/G_v \bigcap P$ is isomorphic to
the quadric $\bQ^{10}$.

For $w = s_{3}s_{7}s_{6}s_{5}s_{4}s_{1}s_{3}s_{2}s_{4}s_{5}s_{6}s_{7}$, 
we have $\hw = (s_{3}s_{1},v)$, where $v$ is as above. 
Therefore, the morphism 
$\hf :\hX(\hw) \longrightarrow X(s_{3}s_{1}) \simeq \bP^{2}$ 
is a fibration with fiber $\bQ^{10}$. Also,
$\hw' = (s_{7}s_{6}s_{5}, s_{3}s_{4}s_{1}s_{3}s_{2}s_{4}s_{5}s_{6}s_{7})$. 
The simple root $\alpha_{3}$ is not a minuscule root of 
$\Supp(s_{3}s_{4}s_{1}s_{3}s_{2}s_{4}s_{5}s_{6}s_{7})$.
Therefore, $\hX(\hw')$ is singular.

For $w = s_{4}s_{3}s_{7}s_{6}s_{5}s_{4}s_{1}s_{3}s_{2}s_{4}s_{5}s_{6}s_{7}$, 
we have  $\hw = (s_{4}s_{3}s_{1}, v)$. Therefore, the morphism 
$\hf :\hX(\hw) \longrightarrow X(s_{4}s_{3}s_{1}) \simeq \bP^{3}$ 
is again a fibration with fiber $\bQ^{10}$. Also, we have
$\hw' = (s_{7}s_{6}, s_{4}s_{3}s_{5}s_{4}s_{1}s_{3}s_{2}s_{4}s_{5}s_{6}s_{7})$. 
The simple root $\alpha_{4}$ is not a minuscule root of 
$\Supp(s_{4}s_{3}s_{5}s_{4}s_{1}s_{3}s_{2}s_{4}s_{5}s_{6}s_{7})$. 
Therefore, $\hX(\hw')$ is singular.

For $w = s_{2}s_{4}s_{3}s_{7}s_{6}s_{5}s_{4}s_{1}s_{3}s_{2}s_{4}s_{5}s_{6}s_{7}$, 
we have $\hw = (s_{2}s_{4}s_{3}s_{1},v)$. Therefore, the morphism 
$\hf : \hX(\hw) \longrightarrow X(s_{2}s_{4}s_{3}s_{1}) \simeq \bP^{4}$ is still 
a fibration with fiber $\bQ^{10}$. Also, we have
$\hw' = (s_{7}s_{6}, s_{2}s_{4}s_{3}s_{5}s_{4}s_{1}s_{3}s_{2}s_{4}s_{5}s_{6}s_{7})$. 
The simple root $\alpha_{2}$ is not a minuscule root of 
$\Supp(s_{2}s_{4}s_{3}s_{5}s_{4}s_{1}s_{3}s_{2}s_{4}s_{5}s_{6}s_{7})$. 
Therefore, $\hX(\hw')$ is singular.

For $w = s_{5}s_{4}s_{3}s_{7}s_{6}s_{5}s_{4}s_{1}s_{3}s_{2}s_{4}s_{5}s_{6}s_{7}$,
we have  $\hw = (s_{5}s_{4}s_{3}s_{1}, v)$. Therefore, the morphism 
$\hf : \hX(\hw) \longrightarrow X(s_{2}s_{4}s_{3}s_{1}) \simeq \bP^{4}$ 
is a fibration with fiber $\bQ^{10}$ once more. Also, we have
$\hw' = (s_{7}, s_{5}s_{4}s_{3}s_{6}s_{5}s_{4}s_{1}s_{3}s_{2}s_{4}s_{5}s_{6}s_{7})$. 
The simple root $\alpha_{5}$ is not a minuscule root of 
$\Supp(s_{5}s_{4}s_{3}s_{5}s_{4}s_{1}s_{3}s_{2}s_{4}s_{5}s_{6}s_{7})$. 
Therefore, $\hX(\hw')$ is singular.

For 
$w =s_{5}s_{2}s_{4}s_{3}s_{7}s_{6}s_{5}s_{4}s_{1}s_{3}s_{2}s_{4}s_{5}s_{6}s_{7}$, 
there are three peak elements. For the standard ordering, we have  
$\hw = (s_{5}, s_{2}s_{4}s_{3}s_{1}, s_{7}s_{6}s_{5}s_{4}s_{3}s_{2}s_{4}s_{5}s_{6}s_{7})$. 
Thus, $\hX(\hw)$ is a tower of fibrations with fibers $\bP^{1}$, $\bP^{4}$ and 
$G_{v}/G_{v} \bigcap P$, where $v$ is as above. 
There is an ordering for which 
$\hw' = (s_{2}, s_{5}s_{4}s_{3}s_{1}, s_{7}s_{6}s_{5}s_{4}s_{3}s_{2}s_{4}s_{5}s_{6}s_{7})$. 
In this case also, $\hX(\hw')$ is a tower of fibrations with fibers 
$\bP^{1}$, $\bP^{4}$ and $G_{v}/G_{v} \bigcap P$. 
For all other orderings, $\hX(\hw')$ is singular.

Finally, if  $\hX(\hw)$ is singular for some $w \in W^I$, then 
$\hX(\hw')$ is also singular for any generalised reduced decomposition 
$\hw'$ obtained by construction 1 corresponding to a nonstandard ordering.

\subsection{Equality of Weyl groups and root inequality in exceptional types}
\label{subsec:except}

Let $G$ be of type $E_6$ or $E_7$. As in Section \ref{sec:weyl},
we consider a minuscule parabolic subgroup $P = P_I$ of $G$,
a Weyl group element $w \in W^I$, and a generalised reduced decomposition 
$\hw=(w_1, w_2, \ldots , w_m)$ of $w$, obtained by construction 1
corresponding to any ordering of $\Peaks(Q_w)$ such that $\hX(\hw)$ 
is smooth. Then Theorem \ref{thm:WMG} adapts to this setting:

\begin{theorem}\label{thm:WME}
With the above notation, we have 
\[ W(G_{w_1, \hx}, T_{w_1}) = W(G_{w_1, w_1 x_1}, T_{w_1}) \]
and this group (viewed as a subgroup of $W$) is generated by simple reflections. 
\end{theorem}

\begin{proof}
By descriptions of $\hX(\hw)$ in \ref{subsec:E6} and \ref{subsec:E7},
we see that the Dynkin diagram of $G_{w_{1}}$ is of type $A$.
Therefore, we can adapt the proof of Proposition \ref{propWMGD}.
\end{proof}

Also, by using the same descriptions, one may readily check the following:

\begin{proposition}\label{prop:heightE}
Let $\alpha$ be the unique simple root such that $w_1^{-1}(\alpha)$ 
is a negative root (see \cite[Lem.~5.1]{BK}). Then we have 
$w_1^{-1}(\alpha) \geq  w^{-1}(\alpha)$. 
Further, $w_1^{-1}(\alpha) = w^{-1}(\alpha)$ if and only if $w = w_1$; 
that is, $m=1$.
\end{proposition}

\end{document}